\documentclass[a4paper,10pt,oneside]{amsart}

\usepackage{amsmath}
\usepackage{amssymb}
\usepackage{amsthm}
\usepackage{latexsym,mathrsfs,bbold}
\usepackage{graphicx}

\newcommand{\wn}{\includegraphics{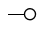}}
\newcommand{\bigo}{\mathcal{O}}

\newtheorem{theorem}{Theorem}[section]
\newtheorem{proposition}[theorem]{Proposition}
\newtheorem{lemma}[theorem]{Lemma}

\theoremstyle{definition}
\newtheorem*{definition}{Definition}

\begin{document}

\title[Planar Maps with Prescribed Degrees]{Limit Laws of Planar Maps\\ with Prescribed Vertex Degrees}
\author{Gwendal Collet}
\address{
Institute of Discrete Mathematics and Geometry\\ 
TU Wien\\
Wiedner Hauptstra\ss{}e 8--10/104\\
1040 Wien\\
Austria
}
\email{gwendal.collet@tuwien.ac.at}

\author{Michael Drmota}
\address{
Institute of Discrete Mathematics and Geometry\\ 
TU Wien\\
Wiedner Hauptstra\ss{}e 8--10/104\\
1040 Wien\\
Austria
}
\email{drmota@tuwien.ac.at}

\author{Lukas Daniel Klausner}
\address{
Institute of Discrete Mathematics and Geometry\\ 
TU Wien\\
Wiedner Hauptstra\ss{}e 8--10/104\\
1040 Wien\\
Austria
}
\email{lukas.d.klausner@tuwien.ac.at}

\thanks{Partially supported by the Austrian Science Fund (FWF) project SFB F50-02 ``Shape Characteristics of Planar Maps and Planar Graphs''.}



\begin{abstract}
We prove a general multi-dimensional central limit theorem for the expected number of vertices of a given degree in the family of planar maps whose vertex degrees are restricted to an arbitrary (finite or infinite) set of positive integers $D$. Our results rely on a classical bijection with mobiles (objects exhibiting a tree structure), combined with refined analytic tools to deal with the systems of equations on infinite variables that arise. We also discuss possible extensions to maps of higher genus and to weighted maps.
\end{abstract}

\maketitle

\section{Introduction and Results}

In this paper we study statistical properties of planar maps, which
are connected planar graphs, possibly with loops and multiple edges, together with
an embedding into the plane. Such objects are frequently used to describe 
topological features of geometric arrangements in two or three spatial dimensions.
Thus, the knowledge of the structure and of properties of ``typical'' objects
may turn out to be very useful in the analysis of particular algorithms that operate on planar maps.
We also want to emphasise the interactions with other fields such as
statistical physics, probability theory, limiting continuous objects and algebraic geometry.
We say that a map is {\it rooted} if an edge $e$ is distinguished and oriented.
It is called the root edge. 
The first vertex $v$ of this oriented edge is called the root vertex. The face to the right of $e$ is called the
root face and is usually taken as the outer (or infinite)
face. Similarly, we call a planar map {\it pointed} if just a vertex $v$ is distinguished.
However, we have to be really careful with the model. In rooted maps the root edge
{\it destroys} potential symmetries, which is not the case if we consider pointed maps.

The enumeration of rooted maps is a classical subject, initiated by Tutte in
the 1960s, see \cite{Tutte}. Among many other results, Tutte computed the number
$M_n$ of rooted maps with $n$ edges, proving the formula
\[
M_n = \frac{2(2n)!}{(n+2)! n!} 3^n 
\]
which directly provides the asymptotic formula
\[
M_n \sim \frac 2{\sqrt \pi} n^{-5/2} 12^n.
\]
We are mainly interested in planar maps with degree restrictions.
Actually, it turns out that the subexponential part $n^{-5/2}$ of the
asymptotic expansion is quite universal and hence to a certain extent
describes the ``physics'' of the combinatorial object. Furthermore, there is always a (very general) 
central limit theorem for the number of vertices of given degree.

\begin{theorem}\label{Th1}
Suppose that $D$ is an arbitrary set of positive integers but not a
subset of $\{1,2\}$. Let $\mathcal{M}_D$ be the class of planar rooted maps with the property
that all vertex degrees are in $D$ and let $M_{D,n}$ denote the number of
maps in $\mathcal{M}_D$ with $n$ edges. 
Furthermore, if $D$ contains only even numbers, then set $\overline d = {\rm gcd}\{i: 2i \in D\}$;
set $\overline d = 1$ otherwise.

Then there exist positive constants
$c_D$ and $\rho_D$ with
\begin{equation}\label{eqTh11}
M_{D,n} \sim c_D n^{-5/2} \rho_D^{-n}, \qquad n \equiv 0 \bmod \overline d.
\end{equation}
Furthermore, let $X_n^{(d)}$ denote the random variable counting vertices of degree $d$ ($\in D$)
in maps in $\mathcal{M}_D$. Then $\mathbb{E}(X_n^{(d)}) \sim \mu_d n$ and $\mathbb{V}(X_n^{(d)}) \sim \sigma_d^2 n$
for some constants $\mu_d > 0$ and $\sigma_d \geq 0$ and for $n\equiv 0 \bmod \overline d$,  and the (possibly infinite) random vector ${\bf X}_n = (X_n^{(d)})_{d\in D}$
($n \equiv 0 \bmod \overline d$) satisfies a central limit theorem, that is, 
\begin{equation}\label{eqTh12}
\frac 1{\sqrt n} \left( {\bf X}_n - \mathbb{E}({\bf X}_n) \right),  \qquad n \equiv 0 \bmod \overline d, 
\end{equation}
converges weakly to a centered Gaussian random variable ${\bf Z}$ (in $\ell^2$).
\end{theorem}

Note that maps where all vertex degrees are $1$ or $2$ are very easy to 
characterise and are not really of interest, and that actually, their asymptotic
properties are different from the general case. It is therefore natural to
assume that $D$ is not a subset of $\{1,2\}$. Also note that for a given $D$,
the constants can be computed to an arbitrary standard of precision.

Since we can equivalently consider dual maps, this kind of problem is the
same as the one considering planar maps with restrictions on the face valencies.
This means that the same results hold if we replace {\it vertex degree} by
{\it face valency}. 
For example, if we assume that all face valencies equal $4$, then we just consider
planar quadrangulations (which have also been studied by Tutte \cite{Tutte}).
In fact, our proofs will refer just to face valencies.

Theorem~\ref{Th1} goes far beyond known results.
There are some general results for the Eulerian case where all vertex 
degrees are even. First, the asymptotic expansion (\ref{eqTh11}) is known
for Eulerian maps by 
Bender and Canfield \cite{BC}. Furthermore, a central limit theorem of the
form (\ref{eqTh12}) is known for all Eulerian maps (without degree restrictions)
\cite{DGM}. However, in the non-Eulerian case there are almost no results of
this kind; there is only a one-dimensional central limit theorem for $X_n^{(d)}$
for all planar maps \cite{DP}.

The uniform distribution of planar maps according to the number of edges is not the only distribution that has been studied. 
Many probabilistic results on planar maps have also be extended to other probability distributions, based on $q$-Boltzmann maps.
Let $\textbf{q}=(q_1,q_2,\dots)$ be a sequence of non-negative weights. A $\textbf{q}$-\emph{Boltzmann map} is a random planar map 
with arbitrary vertex degrees, where the probability of choosing a given map $M$ is proportional to 
$$\prod_{i>0} q_i^{\# \textrm{vertices of degree i in }M}.$$
When such a procedure describes a well-defined probability distribution, $\textbf{q}$ is called \emph{admissible}. 
(We could equivalently use weights of the form 
$\prod q_i^{\# \textrm{faces of degree i in }M}$ by duality.)

In~\cite{MM07,MW08}, the authors showed that, under some integrability conditions, random $\textbf{q}$-Boltzmann maps 
have the same profile as random uniform planar maps. In this spirit we will show that 
(under certain conditions) Theorem~\ref{Th1} also applies to $\textbf{q}$-Boltzmann maps.

\begin{theorem}\label{Th1.2}
Let $\textbf{q}=(q_1,q_2,\dots)$ be a weight sequence with $q_i = \Theta( i^{\alpha} )$ for 
some $\alpha \ge - \frac 32$ and consider corresponding $\textbf{q}$-{Boltzmann maps}.  
Furthermore, let $Y_n^{(d)}$ denote the random variable counting vertices of degree $d$.
Then $\mathbb{E}(Y_n^{(d)}) \sim \mu_d n$ and $\mathbb{V}(Y_n^{(d)}) \sim \sigma_d^2 n$
for some constants $\mu_d > 0$ and $\sigma_d \geq 0$, and
the infinite random vector ${\bf Y}_n = (Y_n^{(d)})_{d\ge 1}$
satisfies a central limit theorem.
\end{theorem}

Again, for a given $\textbf{q}=(q_1,q_2,\dots)$, the constants can be computed to an arbitrary standard of precision.
\bigskip

Graphs can also be embedded on different surfaces than the plane. Given a non-negative integer $g$, a map of genus $g$ is then a connected graph with a proper embedding (where any face is simply connected) on the torus with $g$ holes. In this setting, planar maps, drawn on the plane (or equivalently, on the sphere), are simply maps of genus $0$. The first results in higher genus were obtained by Bender and Canfield~\cite{BC86}, providing the asymptotic number of rooted maps of genus $g$ with $n$ edges:
\[
	M_n^{(g)} \sim t_g n^{5(g-1)/2} 12^n.
\]
These asymptotics were later rederived via bijective methods by Chapuy, Marcus and Schaeffer~\cite{CMS07}, leading to numerous developments in the study of maps on any surface in recent years.

In this context, we will establish a generalisation of Theorem~\ref{Th1} in the bipartite case.

\begin{theorem}\label{Th3}
Suppose that $D \neq \{2\}$ is an arbitrary set of positive even integers, let $\mathcal{M}^{(g)}_D$ be the class of rooted bipartite maps of genus $g$ with the property
that all vertex degrees are in $D$ and let $M^{(g)}_{D,n}$ denote the number of
maps in $\mathcal{M}^{(g)}_D$ with $n$ edges. 
Furthermore, set $\overline d = {\rm gcd}\{i: 2i \in D\}$.

Then there exist positive constants
$c^{(g)}_D$ and $\rho^{(g)}_D$ with
\begin{equation}\label{eqTh31}
M^{(g)}_{D,n} \sim c^{(g)}_D n^{5(g-1)/2} (\rho^{(g)}_D)^{-n}, \qquad n \equiv 0 \bmod \overline d.
\end{equation}

Furthermore, let $X_n^{(d)}$ denote the random variable counting vertices of degree $d$ ($\in D$)
in maps in $\mathcal{M}^{(g)}_D$. Then $\mathbb{E}(X_n^{(d)}) \sim \mu_d n$ and $\mathbb{V}(X_n^{(d)}) \sim \sigma_d^2 n$
for some constants $\mu_d > 0$ and $\sigma_d \geq 0$ and for $n\equiv 0 \bmod \overline d$,  and the (possibly infinite) random vector ${\bf X}_n = (X_n^{(d)})_{d\in D}$
($n \equiv 0 \bmod \overline d$) satisfies a central limit theorem.
\end{theorem}
Theorem~\ref{Th1} can be easily recovered for planar bipartite maps by setting $g=0$. The main difference lies in the exponent $5(g-1)/2$, which also appears to be universal for rooted maps of genus $g$. Hence Theorem~\ref{Th3} is expected to hold for any $D$ without restriction.

Theorem~\ref{Th3} covers maps on any orientable surface. For the picture to be complete, one would need to derive a similar result for general surfaces, including non-orientable ones (for instance, the projective plane). In the article mentioned above~\cite{BC86}, Bender and Canfield also showed similar asymptotics for the number of rooted maps drawn on a non-orientable surface of type $h$, for $h$ any non-negative half-integer:
\[
	M_n^{(h)} \sim p_g n^{5(h-1)/2} 12^n.
\]
This result has also been rederived bijectively by Chapuy and Do\l{}\c{e}ga~\cite{CD} using some \emph{local} orientations of the surface, but the bijection at play no longer preserves degrees. On the other hand, the key bijection that we will use throughout this work was recently extended by Bettinelli~\cite{Bett} to non-orientable surfaces, but in this case, the family of objects in bijection seem much harder to describe and to enumerate.

\bigskip

Section~\ref{sec2} introduces planar mobiles which, being in bijection with pointed planar maps, will reduce our analysis to simpler objects with a tree structure. Their asymptotic behaviour is derived in Section~\ref{sec3}, first for the simpler case of bipartite maps (i.~e.\ when $D$ contains only even integers), then for families of maps without constraints on $D$. Sections~\ref{sec4} and \ref{sec4.2} are devoted to the proof of the central limit theorem using analytic tools from~\cite{Drmotabook,DGM}. Finally, in Section~\ref{sec5} we discuss combinatorics and asymptotics of bipartite maps on orientable surfaces of higher genus. The expressions we obtain are much more involved than in the planar case, but we obtain similar analytic results.

\section{Planar Mobiles}\label{sec2}

Instead of investigating planar maps themselves, we will follow the principle presented by Chapuy, Fusy, Kang, and Shoilekova in~\cite{CFKS}, whereby pointed planar maps are bijectively related to a certain class of trees called mobiles. (Their version of mobiles differs from the definition originally given in \cite{BDFG}; the equivalence of the two definitions is not shown explicitly in \cite{CFKS}, but \cite{CF} gives a straightforward proof.)

\begin{definition}\label{def:mobiles}
A \emph{mobile} is a planar tree -- that is, a map with a single face -- such that there are two kinds of vertices (black and white), edges only occur as black--black edges or black--white edges, and black vertices additionally have so-called ``legs'' attached to them (which are not considered edges), whose number equals the number of white neighbour vertices. A \emph{bipartite mobile} is  a mobile without black--black edges. The \emph{degree} of a black vertex is the number of half-edges plus the number of legs that are 
attached to it. A mobile is called \emph{rooted} if an edge is distinguished and oriented.
\end{definition}

\begin{figure}
\centering
\includegraphics[width=0.45\linewidth]{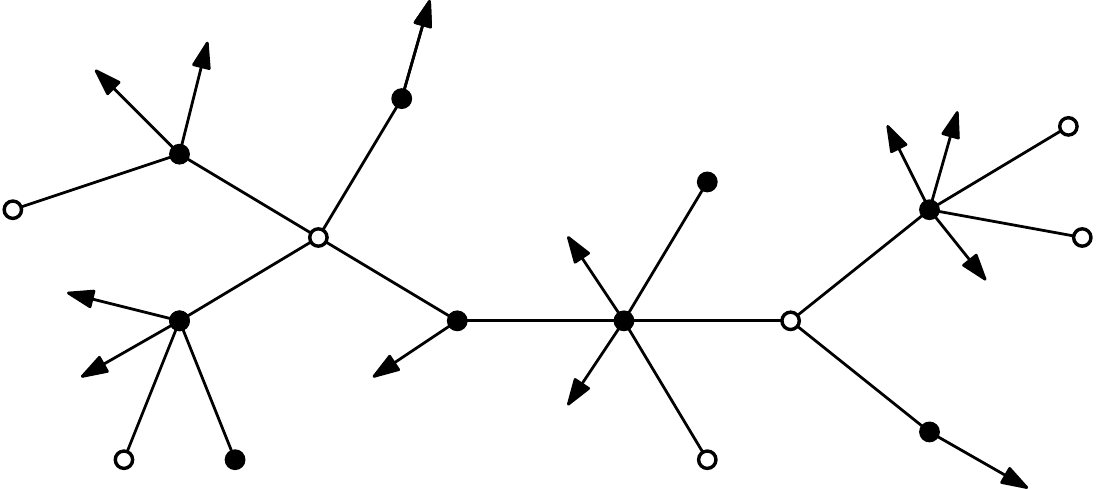}
\caption{An example of a planar mobile.}
\label{fig:mobiles-rec}
\end{figure}

The essential observation is that mobiles are in bijection to pointed planar maps.

\begin{theorem}\label{Th2}
There is a bijection between mobiles that contain at least one black vertex and pointed planar maps, where white vertices in the mobile correspond to non-pointed vertices in the equivalent planar map, black vertices correspond to faces of the map, and the degrees of the black vertices 
correspond to the face valencies. This bijection induces a bijection on the edge sets such that
the number of edges is the same. (Only the pointed vertex of the map has no counterpart.)

Similarly, rooted mobiles that contain at least one black vertex are in bijection to rooted and vertex-pointed planar maps.

Finally, bipartite mobiles with at least two vertices correspond to bipartite maps with at least two vertices,
in the unrooted as well as in the rooted case.
\end{theorem}

\begin{proof}
For the proof of the bijection between mobiles and pointed maps we refer to \cite{CF}, where the
bipartite case is also discussed. It just remains to note that the induced bijection on the edges
can be directly used to transfer the root edge together with its direction.
\end{proof}

\subsection{Bipartite Mobile Counting}

We start with bipartite mobiles since they are more easy to count, 
in particular if we consider rooted bipartite mobiles, see \cite{CF}.

\begin{proposition}\label{Pro1}
Let $R \equiv R(t,z,x_1,x_2,\ldots)$ be the solution of the equation
\begin{equation}\label{eqPro11}
R = tz + z\sum_{i\ge 1} x_{2i} {2i-1 \choose i} R^i.
\end{equation}
Then the generating function $M \equiv M(t,z,x_1,x_2,\ldots)$ of bipartite rooted maps satisfies
\begin{equation}\label{eqPro12}
\frac{\partial M}{\partial t} = 2\left( R/z - t\right),
\end{equation}
where the variable $t$ corresponds to the
number of vertices, $z$ to the number of edges, and $x_{2i}$, $i\ge 1$, to the
number of faces of valency $2i$. 
\end{proposition}

\begin{proof}
Since rooted mobiles can be considered as ordered rooted trees
(which means that the neighbouring vertices of the root vertex 
are linearly ordered and the subtrees rooted at these neighbouring vertices
are again ordered trees), we can describe them recursively.
This directly leads to a functional equation for $R$ of the form
\[
R = \frac{tz}{1- z \sum_{i\ge 1} x_{2i} {2i-1 \choose i} R^{i-1} }
\]
which is apparently the same as (\ref{eqPro11}). Note that the factor ${2i-1 \choose i}$ is
precisely the number of ways of grouping $i$ legs and $i-1$ edges around a black vertex
(of degree $2i$; one edge is already there).

Hence, the generating function of rooted mobiles that are rooted by a white vertex is 
given by $R/z$. Since we have to discount the mobile that consists just of one (white)
vertex, the generating function of rooted mobiles that are
rooted at a white vertex and contain at least two vertices is given by
\begin{equation}\label{eqwb}
R/z - t = \sum_{i\ge 1} x_{2i} {2i-1 \choose i} R^i.
\end{equation}
We now observe that the right-hand side of (\ref{eqwb}) is precisely the
generating function of rooted mobiles that are rooted at a black vertex
(and contain at least two vertices). Summing up, the generating function of
bipartite rooted mobiles (with at least two vertices) is given by
$
2 (R/z -t).
$
Finally, if $M$ denotes the generating function of bipartite rooted maps 
(with at least two vertices) then 
$\frac{\partial M}{\partial t}$ corresponds to rooted maps where a non-root vertex is
pointed (and discounted). Thus, by Theorem~\ref{Th2} we obtain (\ref{eqPro12}).
\end{proof}
It is clear that Formula~\eqref{eqPro12} can be specialised to count $M_D$ for any subset $D$ of even positive integers: 
It suffices to set $x_{2i} = 1$ for $2i\in D$ and $x_{2i} = 0$ otherwise.

\subsection{General Mobile Counting}

We now proceed to develop a mechanism for general mobile counting that is
adapted from \cite{CFKS}. For this, we will require Motzkin paths.
A \emph{Motzkin path} is a path starting at $0$ and going rightwards for a number of steps; the steps are either diagonally upwards ($+1$), straight ($0$) or diagonally downwards ($-1$). A \emph{Motzkin bridge} is a Motzkin path from $0$ to $0$. A \emph{Motzkin excursion} is a Motzkin bridge which stays non-negative.

We define generating functions in the variables $t$ and $u$, which count the number of steps of type $0$ and $-1$, respectively. (Explicitly counting steps of type $1$ is then unnecessary, of course.) The ordinary generating functions of Motzkin bridges, Motzkin excursions, and Motzkin paths from $0$ to $+1$ shall be denoted by $B(t,u)$, $E(t,u)$ and $B^{(+1)}(t,u)$, respectively. 
By decomposing these three types of paths by their last passage through $0$, we arrive at the equations (compare with \cite{CFKS}):
\[
E = 1 + tE + uE^2, \quad 
B = 1 + (t + 2uE)B, \quad
B^{(+1)} = EB.
\]
In what follows we will also make use of bridges where the first step is either of type $0$ or $-1$.
Clearly, their generating function $\overline B$ is given by
$ \overline B = t B + u B^{(+1)} = B(t+uE)$.

When Motzkin bridges are not constrained to stay non-negative, they can be seen as an arbitrary arrangement of a given number of steps $+1,0,-1$. It is then possible to obtain explicit expressions for
\begin{align}
B_{\ell,m} &= [t^\ell u^m] B(t,u) = \binom{l+2m}{l,m,m},  \label{eqB1} \\
B_{\ell,m}^{(+1)} &= [t^\ell u^m] B^{(+1)}(t,u) = \binom{l+2m+1}{l,m,m+1}, \label{eqB2} \\
\overline B_{\ell,m} &= [t^\ell u^m] \overline B(t,u) =  B_{\ell-1,m} + B_{\ell,m-1}^{(+1)} = \frac{l+m}{l+2m}\binom{l+2m}{l,m,m}.  \label{eqB3}
\end{align}

Using the above, we can now finally compute relations for generating functions of proper classes of mobiles. We define the following series, where $t$ corresponds to the
number of white vertices, $z$ to the number of edges, and $x_i$, $i\ge 1$, to the
number of black vertices of degree $i$:
\begin{itemize}
\item $L(t,z,x_1,x_2,\ldots)$ is the series counting rooted mobiles that are rooted at a black vertex and where an additional edge is attached to the black vertex. 
\item $R(t,z,x_1,x_2,\ldots)$ is the series counting rooted mobiles that are rooted at a white vertex and  where an additional edge is attached to the root vertex. 
\end{itemize}
Similarly to the above we obtain the following equations for the generating functions of mobiles and rooted maps.

\begin{figure}
\centering
\includegraphics[width=0.8\linewidth]{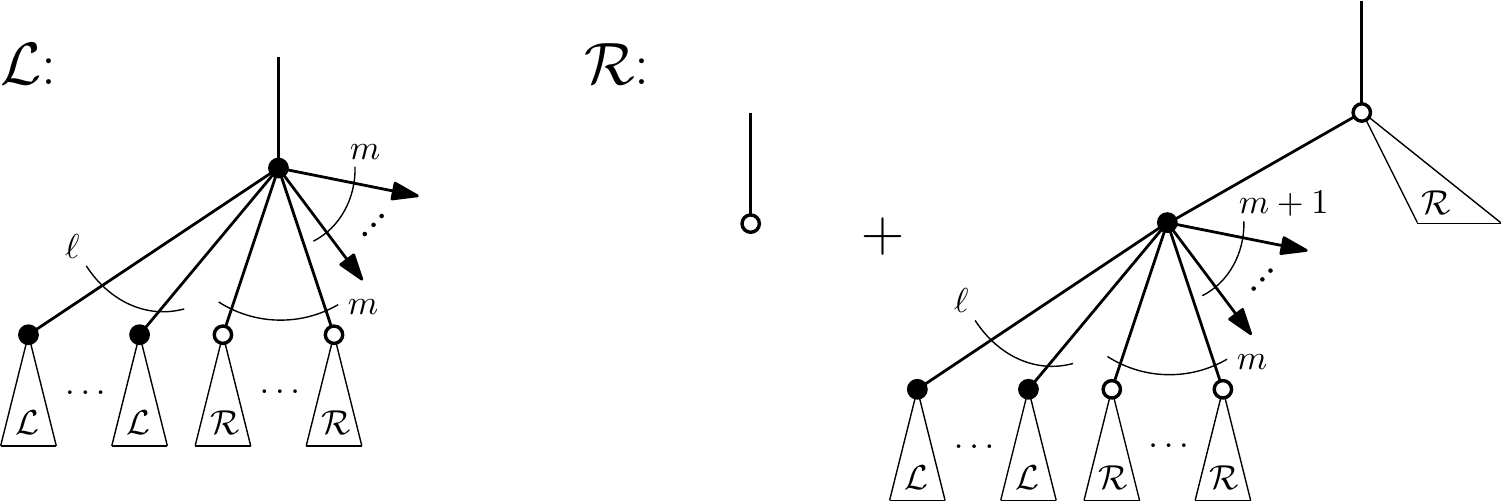}
\caption{Recursive description of mobiles leading to the system of equations~\eqref{eqPro21}.}
\label{fig:mobiles-rec}
\end{figure}

\begin{proposition}\label{Pro2}
Let $L \equiv L(t,z,x_1,x_2,\ldots)$
and 
$R \equiv R(t,z,x_1,x_2,\ldots)$ be the solutions of the system of equations
\begin{equation}
\begin{aligned}
L &= z \sum_{\ell,m} x_{2m+\ell+1} B_{\ell,m} L^\ell R^m,  \\
R &= tz + z \sum_{\ell,m} x_{\ell+2m+2} B_{\ell,m}^{(+1)} L^\ell R^{m+1}, 
%
\end{aligned}\label{eqPro21} 
\end{equation}
and let $T \equiv T(t,z,x_1,x_2,\ldots)$ be given by
\begin{equation}\label{eqPro21bis}
T = 1 + \sum_{\ell,m} x_{2m+\ell} \overline B_{\ell,m} L^\ell R^m,
\end{equation}
where the numbers $B_{\ell,m}$, $B_{\ell,m}^{(+1)}$, and $\overline B_{\ell,m}$ are given by (\ref{eqB1})--(\ref{eqB3}).
Then the generating function $M \equiv M(t,z,x_1,x_2,\ldots)$ of rooted maps satisfies
\begin{equation}\label{eqPro22}
\frac{\partial M}{\partial t} = R/z-t + T,
\end{equation}
where the variable $t$ corresponds to the
number of vertices, $z$ to the number of edges, and $x_i$, $i\ge 1$, to the
number of faces of valency $i$. 
\end{proposition}

\begin{proof}
The system (\ref{eqPro21}) is just a rephrasement of the recursive structure of rooted
mobiles. Note that the numbers $B_{\ell,m}$ and $B_{\ell,m}^{(+1)}$ are used to 
count the number of ways to circumscribe a specific black vertex and considering white vertices, black vertices and ``legs'' as steps $-1$, $0$ and $+1$. The generating function $T$ given in (\ref{eqPro21bis}) is
then the generating function of rooted mobiles where the root vertex is black.

Finally, the equation (\ref{eqPro22}) follows from Theorem~\ref{Th2} since $R/z-t$ corresponds
to rooted mobiles with at least one black vertex where the root vertex is white and $T$ corresponds
to rooted mobiles where the root vertex is black.
\end{proof}


\section{Asymptotic Enumeration}\label{sec3}

In this section we prove the asymptotic expansion (\ref{eqTh11}). 
It turns out that it is much easier to start with bipartite maps.
Actually, the bipartite case has already been treated by Bender and Canfield \cite{BC}.
However, we apply a slightly different approach, which will then be extended
to cover the general case as well the central limit theorem.

\subsection{Bipartite Maps}

Let $D$ be a non-empty subset of even positive integers different from $\{2\}$.
Then by Proposition~\ref{Pro1} the counting problem reduces to the discussion 
of the solutions $R_D \equiv R_D(t,z)$ of the functional equation
\begin{equation}\label{eqPro11-2}
R_D = tz + z\sum_{2i\in D} {2i-1 \choose i} R_D^i
\end{equation}
and the generating function $M_D(t,z)$ that satisfies the relation
\begin{equation}\label{eqPro12-2}
\frac{\partial M_D}{\partial t} = 2\left( R_D/z - t\right).
\end{equation}

Let $\overline d = {\rm gcd}\{i : 2i \in D\}$. Then for combinatorial reasons it 
follows that there only exist maps with $n$ edges for $n$ that are divisible by $\overline d$.
This is reflected by the fact that the equation (\ref{eqPro11-2}) can we rewritten
in the form
\begin{equation}\label{eqPro11-2-1}
\widetilde R = t + \sum_{2i\in D} {2i-1 \choose i} z^{i/\overline d} \widetilde R^i,
\end{equation}
where we have substituted $R_D(t,z) = z \widetilde R(t,z^{\overline d})$. (Recall that we finally 
work with $R_D/z$.) 

\begin{lemma}\label{Le1}
There exists an analytic function $\rho_D(t)$ with $\rho_D(1) > 0$ and $\rho_D'(1) \ne 0$
that is defined in a neighbourhood of $t=1$, and there exist analytic functions 
$g(t,z)$, $h(t,z)$ with $h(1,\rho_D(1))> 0$ that are defined in 
a neighbourhood of $t=1$ and $z=\rho_D(1)$ such that the unique solution $R_D \equiv R_D(t,z)$ of the
equation (\ref{eqPro11-2}) that is analytic at $z=0$ and $t=0$ can be represented as
\begin{equation}\label{eqLe1}
R_D = g(t,z) - h(t,z) \sqrt{1 - \frac z{\rho_D(t)}}.
\end{equation}
Furthermore, the values $z = \rho_D(t) e(2\pi i j/\overline d)$, $j\in \{0,1,\ldots,\overline d-1\}$, are the
only singularities of the function $z\mapsto R_D(t,z)$ on the disc $|z| \le \rho_D(t)$,
and for some sufficiently small $\eta>0$ there exists an analytic continuation of $R_D$ to the range 
$|z| < |\rho_D(t)| + \eta$, $\arg(z- \rho_D(t) e(2\pi i j/\overline d)) \ne 0$, $j\in \{0,1,\ldots,\overline d-1\}$.
\end{lemma}

\begin{proof}
From general theory (see \cite[Theorem 2.21]{Drmotabook}) we know that an equation of the
form $R = F(t,z,R)$, where $F$ is a power series with non-negative coefficients, 
has -- usually -- a square-root singularity of the form (\ref{eqLe1}).
We only have to assume that the function $R\to F(t,z,R)$ is neither constant nor a linear polynomial
and that there exist solutions $z_0> 0$, $R_0 > 0$ of the system of equations
\[
R_0 = F(1,z_0,R_0), \qquad 1 = F_R(1,z_0,R_0)
\]
which are inside the range of convergence of $F$. Furthermore, we have to assume that
$F_z(1,z_0,R_0) > 0$ and $F_{RR}(1,z_0,R_0) > 0$  to ensure that (\ref{eqLe1}) holds not
only for $t=1$ but in a neighbourhood of $t=1$,  and the condition $F_t(1,z_0,R_0) > 0$ ensures that
$\rho_D'(1) \ne 0$.

This means that in our case we have to deal with the system of equations
\[
R_0 = z_0 + z_0 \sum_{2i\in D} {2i-1 \choose i} R_0^i, \quad
1 = z_0 \sum_{2i\in D} i{2i-1 \choose i} R_0^{i-1},
\]
or just with a single equation (after eliminating $z_0$)  
\begin{equation}\label{eqR0}
\sum_{2i\in D} (i-1){2i-1 \choose i} R_0^i = 1.
\end{equation}
It is clear that (\ref{eqR0}) has a unique positive solution if $D$ is finite.
(We also recall that all $i\ge 1$, since $2i$ has to be positive.)
If $D$ is infinite, we have to be more precise. Actually, we will show that 
(\ref{eqR0}) has a unique positive solution $R_0 < 1/4$. This follows from 
the fact that 
\[
(i-1) {2i-1 \choose i} \sim \frac{4^i \sqrt i} {2 \sqrt{\pi}}.
\]
Thus, if $D$ is infinite, it follows that the power series 
$
x\mapsto H(x) = \sum_{2i\in D} (i-1){2i-1 \choose i} x^i
$
has radius of convergence $1/4$ and we also have $H(x)\to\infty$ as
$x\to 1/4-$ since each non-zero term satisfies
\[
\lim_{x\to 1/4} (i-1) {2i-1 \choose i} x^i \sim \frac{\sqrt i} {2 \sqrt{\pi}},
\]
which is unbounded for $i\to\infty$. 
Finally, we set
\[
\rho_D(1)= z_0 = \left( \sum_{2i\in D} i{2i-1 \choose i} R_0^{i-1} \right)^{-1}.
\]

It is clear that $F_z(1,z_0,R_0) > 0$, $F_{RR}(1,z_0,R_0) > 0$, and  $F_t(1,z_0,R_0) > 0$.
Hence we obtain the representation (\ref{eqLe1}) in a neighbourhood of
$z = z_0 = \rho_D(1)$ and $t = 1$.

Next, let us discuss the analytic continuation property. 
If $\overline d = {\rm gcd}\{i : 2i \in D\} = 1$ then it follows from the equation 
(\ref{eqPro11-2}) that the coefficients $[z^n] R_D(1,z)$ are positive for
$n\ge n_0$ (for some $n_0$). Consequently \cite[Theorem 2.21]{Drmotabook}
(see also \cite[Theorem 2.16]{Drmotabook}) implies that for some sufficiently small $\eta>0$
there is an analytic continuation to the region
$|z| < |\rho_D(t)| + \eta$, $\arg(z- \rho_D(t)) \ne 0$. 
If $\overline d> 1$, then we can first reduce equation (\ref{eqPro11-2}) to a
an equation (\ref{eqPro11-2-1}) for the function $\widetilde R$ that is given
by $R_D(t,z) = z \widetilde R(t,z^{\overline d})$. We now apply the above method to this
equation and obtain corresponding properties for $\widetilde R$. 
Of course, these properties directly translate to $R_D$, and we are done.
\end{proof}

It is now relatively easy to obtain similar properties for
$M_D(t,z)$.

\begin{lemma}\label{Le2}
The function $M \equiv M_D(t,z)$ that is given by (\ref{eqPro12-2}) has the representation 
\begin{equation}\label{eqLe2}
M_D = g_2(t,z) + h_2(t,z) \left(1 - \frac z{\rho_D(t)} \right)^{3/2}
\end{equation}
in a neighbourhood of $t=1$ and $z=\rho_D(1)$, where 
the functions $g_2(t,z)$, $h_2(t,z)$ are analytic in a neighbourhood of $t=1$ and $z=\rho_D(1)$ and we have $h_2(1,\rho_D(1)) > 0$.
Furthermore, the values $z = \rho_D(t) e(2\pi i j/\overline d)$, $j\in \{0,1,\ldots,\overline d-1\}$, are the
only singularities of the function $z\mapsto M_D(t,z)$ on the disc $|z| \le \rho_D(t)$,
and for some sufficiently small $\eta>0$ there exists an analytic continuation of $M_D$ to the range 
$|z| < |\rho_D(t)| + \eta$, $\arg(z- \rho_D(t) e(2\pi i j/\overline d)) \ne 0$, $j\in \{0,1,\ldots,\overline d-1\}$.
\end{lemma}

\begin{proof}
This is a direct application of \cite[Lemma 2.27]{Drmotabook}.
\end{proof}

In particular it follows that $M_D(1,z)$ has the singular representation
of the form (\ref{eqLe2}) with a dominant singularity $\left(1 - z/{\rho_D(1)} \right)^{3/2}$
near $z= \rho_D(1)$. The singular representations are of the same kind 
near $z = \rho_D(1) e(2\pi i j/\overline d)$, $j\in \{1,\ldots,\overline d-1\}$, and we have
the analytic continuation property. Hence it follows by~usual singularity
analysis (see for example \cite[Corollary 2.15]{Drmotabook}) that 
there exists a constant $c_D> 0$ such that
\[
[z^n] M_D(1,z) \sim c_D n^{-5/2} \rho_D(1)^{-n}, \qquad n \equiv 0 \bmod \overline d,
\]
which completes the proof of the asymptotic expansion in the bipartite case.

\subsection{General Maps}

We now suppose that $D$ contains at least one odd number. It is easy to observe that
in this case we have $[z^n] M_D(1,z) > 0$ for $n\ge n_0$ (for some $n_0$), so
we do not have to deal with several singularities.
By Proposition~\ref{Pro2} we have to consider the system of equations
for $L_D \equiv L_D (t,z)$,  $R_D \equiv R_D(t,z)$:
\begin{align}
L_D &= z \sum_{i\in D}  \sum_m B_{i-2m-1,m} L_D^{i-2m-1} R_D^m,  \label{eqPro21-2.0} \\
R_D &= tz + z \sum_{i\in D}  \sum_m B_{i-2m-2,m}^{(+1)} L_D^{i-2m-2} R_D^{m+1},  \label{eqPro21-2}
\end{align}
and also the function
\[
T_D = T_D(t,z) = 1 + \sum_{i\in D} \sum_m \overline B_{i-2m,m} L_D^{i-2m} R_D^m.
\]

\begin{lemma}\label{Le3}
There exists an analytic function $\rho_D(t)$ with $\rho_D(1) > 0$ and $\rho_D'(1) \ne 0$
that is defined in a neighbourhood of $t=1$, and there exist analytic functions 
$g(t,z)$, $h(t,z)$ with $h(1,\rho_D(1))> 0$ that are defined in 
a neighbourhood of $t=1$ and $z=\rho_D(1)$ such that 
\begin{equation}\label{eqLe3}
R_D/z-t + T_D = g(t,z) - h(t,z) \sqrt{1 - \frac z{\rho_D(t)}}.
\end{equation}
Furthermore, the value $z = \rho_D(t)$ is the
only singularity of the function $z\mapsto R_D/z-t + T_D$ on the disc $|z| \le \rho_D(t)$,
and for some sufficiently small $\eta>0$ there exists an analytic continuation of $R_D$ to the range 
$|z| < |\rho_D(t)| + \eta$, $\arg(z- \rho_D(t)) \ne 0$.
\end{lemma}

\begin{proof}
The system of equations (\ref{eqPro21-2.0})--(\ref{eqPro21-2}) -- which we write 
in short-hand notation as $L_D = F(t,z,L_{D} ,R_D)$, $R_{D} = G(t,z,L_D,R_D)$ -- is 
a strongly connected system of two equations such that $F$ and $G$
can be expressed as power series with non-negative coefficients. 
It is known that such a system of equations has in principle the same analytic properties
(including the singular behaviour of its solutions) as a single equation, see 
\cite[Theorem 2.33]{Drmotabook}; however, we have to be sure that the regions of 
convergence of $F$ and $G$ are large enough.

In particular, if $D$ is finite, then we have a positive algebraic system 
and we are done, see \cite{BD}.
In the infinite case we have to argue in a different way.
First of all, it is clear from the explicit solutions of $E = E(t,u) =
(1-t-\sqrt{(1-t)^2 - 4u})/(2u)$ and 
$B = B(t,u) = 1/\sqrt{(1-t)^2 - 4u}$ that $F$ and $G$ (and all their derivatives
with respect to $L_D$ and $R_D$) 
are certainly convergent
if $2 |L_D| - |L_D|^2 + 4 |R_D| < 1$.
On the other hand, it follows similarly to the bipartite case that
the derivatives of $F$ and $G$ are divergent if $L_D > 0$, $R_D> 0$, and
$2 L_D - L_D^2 + 4 R_D = 1$. To see this 
we consider the function 
\begin{align*}
B(ts, us^2) &= \frac 1{\sqrt{1-2ts+t^2s^2-4us^2}} = \sum_{\ell,m} B_{\ell,m} s^{2m+\ell} t^\ell u^m \\
&= \sum_i s^i \sum_m B_{i-2m,m}t^{i-2m} u^m.
\end{align*}
By singularity analysis it follows (for $t,u> 0$) that 
\[
\sum_m B_{i-2m,m} t^{i-2m} u^m \sim c\, i^{-1/2} h(t,u)^{-i},
\]
where $c>0$ and $h=h(t,u)> 0$ satisfies the equation $1-2th+t^2h^2-4uh^2 = 0$. Similarly, we can consider
derivatives of $F$ which correspond, for example, to sums of the form
\[
\sum_m B_{i-2m,m} m t^{i-2m} u^m \sim c' i^{1/2} h(t,u)^{-i}.
\]
In particular, if $h(t,u)= 1$ (which is the case if $2t-t^2-4u= 1$), then this term
diverges for $i\to \infty$.	Thus, the derivatives of $F$ and $G$ diverge if
$L_D > 0$, $R_D> 0$, and
$2 L_D - L_D^2 + 4 R_D = 1$.

In order to determine the singularity of the system 
$L_D = F(t,z,L_D,R_D)$, $R_{D} = G(t,z,L_D,R_D)$
we have to find positive solutions of $L_0,R_0,z_0$ of the system
\begin{equation}\label{eqLe3x}
L_{0} = F(1,z_0,L_{0} ,R_0), \quad
R_{0} = G(1,z_0,L_{0} ,R_0), \quad
1 = \frac{G_{L_D} F_{R_D}}{1 - F_{L_D}} + G_{R_D}.
\end{equation}
We do this in the following way. Starting with $z_0 = 0$, we increase
$z_0$ and solve the first two equations to get $L_0 = L_0(z_0)$, $R_0 = R_0(z_0)$
till the third equation is satisfied. (Note that for $z_0 = 0$, the right-hand side is
$0$ and, thus, smaller than $1$.) As long as the right-hand side of the third equation is
smaller than $1$, it follows from the implicit function theorem that there is a local 
analytic continuation of the solutions $L_0 = L_0(z_0)$, $R_0 = R_0(z_0)$. Furthermore, since
$L_0> 0$ and $R_0> 0$, we have to be in the region of convergence of the derivatives of 
$F$ and $G$, that is, $2 L_{0} - L_{0}^2 + 4 R_0 <1$. From this it also follows that
the solutions $L_0 = L_0(z_0)$, $R_0 = R_0(z_0)$ naturally extend to a point 
where the right-hand side of the third equation equals $1$, so that the above system has
a solution $(1,z_0, L_0,R_0)$. Of course, at this point the derivatives of $F$ and $G$ have to be finite, which
implies that $(1, z_0, L_0,R_0)$ lies inside the region of convergence of $F$ and $G$.

This finally shows that all assumptions of \cite[Theorem 2.33]{Drmotabook} are satisfied.
Thus, singular representation of type (\ref{eqLe3}) and the analytic continuation property
follow for the functions $L_D = L_D(t,z)$, $R_D=R_D(t,z)$.
Hence, the same kind of properties follows for $T_D = T_D(t,z)$ and consequently also for
$R_D/z - t + T_D$.
\end{proof}

Lemma~\ref{Le3} shows that we are precisely in the same situation as
in the bipartite case (actually, it is slightly easier since there is only
one singularity on the circle $|z| = \rho_D(t)$). Hence we immediately
get the same property for $M_D$ as stated in Lemma~\ref{Le2} and
consequently the proposed asymptotic expansion (\ref{eqTh11}).

\section{Central Limit Theorem for Bipartite Maps}\label{sec4}

Based on this previous result, we now extend our analysis
to obtain a central limit theorem. Actually, this is immediate
if the set $D$ is finite, whereas the infinite case 
needs much more care.

Let $D$ be a non-empty subset of even positive integers different from $\{2\}$.
Then by Proposition~\ref{Pro1} the generating functions 
$R_D \equiv R_D(t,z,(x_{2i})_{2i\in D})$ and $M_D \equiv M_D(t,z,$ $(x_{2i})_{2i\in D})$
satisfy the equations
\begin{equation}\label{eqPro11-3}
R_D = tz + z\sum_{2i\in D} x_{2i} {2i-1 \choose i} R_D^i
\quad \mbox{and}\quad
\frac{\partial M_D}{\partial t} = 2\left( R_D/z - t\right).
\end{equation}

If $D$ is finite, then the number of variables is finite, too, and
we can apply \cite[Theorem 2.33]{Drmotabook} to obtain a representation of 
$R_D$ of the form
\begin{equation}\label{eqRgen}
R_D = g(t,z,(x_{2i})_{2i\in D}) - h(t,z,(x_{2i})_{2i\in D}) \sqrt{1 - \frac z{\rho_D(t,(x_{2i})_{2i\in D})}}.
\end{equation}
A proper extension of the
transfer lemma \cite[Lemma 2.27]{Drmotabook} (where the variables $x_{2i}$
are considered as additional parameters) leads to
\begin{equation}\label{eqMgen}
M_D = g_2(t,z,(x_{2i})_{2i\in D}) + h_2(t,z,(x_{2i})_{2i\in D}) \left(1 - \frac z{\rho_D(t,(x_{2i})_{2i\in D})}\right)^{3/2},
\end{equation}
and finally 
\cite[Theorem 2.25]{Drmotabook} implies a multivariate central limit 
theorem for the random vector ${\bf X}_n = (X_n^{(2i)})_{2i\in D}$
of the proposed form.  

Thus, we just have to concentrate on the infinite case. 
Actually, we proceed there in a similar way; however, we have to
take care of infinitely many variables. There is no real problem to derive
the same kind of representation (\ref{eqRgen}) and (\ref{eqMgen}) if $D$ is
infinite. Everything works in the same way as in the finite case,
we just have to assume that the variables $x_i$ are sufficiently close to $1$.
And of course we have to use a proper notion of analyticity in
infinitely many variables. We only have to apply 
the functional analytic extension of the above cited theorems
that are given in \cite{DGM}. Moreover, in order to obtain 
a central limit theorem we need a proper 
adaption of \cite[Theorem~3]{DGM}. This theorem handles the case of
 a single equation $y = F(z,(x_i)_{i\in I},y)$ for a 
generating function $y = y(z,(x_i)_{i\in I})$ that encodes the
distribution of a random vector $(X_n^{(i)})_{i\in I}$ in the form
\[
y = \sum_n y_n \cdot \mathbb{E} \left(\prod_{i\in I} x_i^{X_n^{(i)}} \right) z^n,
\]
where $X_n^{(i)} = 0$ for $i> cn$ (for some constant $c> 0$) which also
implies that all appearing potentially infinite products are in fact finite.
(In our case this is satisfied since there is no vertex of degree larger than $2n$
if we have $n$ edges.) Note that if we let $x_i = e^{st_i}$, then
\[
\mathbb{E}\left(\prod_{i \in I} x_i^{X_n^{(i)}}\right) = \mathbb{E}\left(e^{s \cdot \sum_{i \in I} t_i X_n^{(i)}}\right)
\]
is exactly the moment-generating function of the projected random variable
\[
\sum_{i \in I} t_i X_n^{(i)}.
\]

As we can see from the proof of \cite[Theorem~3]{DGM},
the essential part is to provide tightness of the involved normalised random vector,
and tightness can be checked with the help of moment conditions.
It is clear that asymptotics of moments for $X_n^{(i)}$ can be calculated with
the help of derivatives of $F$, for example $\mathbb{E} X_n^{(i)} = F_{x_i}/(\rho_D F_z)\cdot n + \bigo(1)$.
This follows from the fact all information on the asymptotic behaviour of the moments 
is encoded in the derivatives of the singularity 
$\rho_D(t,z,(x_i)_{i\in I})$, and by implicit differentiation these derivatives 
relate to derivatives of $F$. More precisely, \cite[Theorem~3]{DGM} says that 
the following conditions are sufficient to deduce tightness of the
normalised random vector:
\[
\sum_{i\in I} F_{x_i} < \infty, \qquad \sum_{i\in I} F_{yx_i}^2 < \infty, \qquad \sum_{i\in I} F_{x_ix_i} < \infty,
\]\vspace{-20pt}
\begin{align*}
F_{zx_i} &= o(1), & F_{zx_ix_i} &= o(1), & F_{yyx_i} &= o(1), & F_{yyx_ix_i} &= o(1), \\
F_{zzx_i} &= \bigo(1), & F_{zyx_i} &= \bigo(1), & F_{zyyx_i} &= \bigo(1), & F_{yyyx_i} &= \bigo(1),
\end{align*}
as $i\to \infty$, where all derivatives are evaluated at $(1,\rho_D,(1)_{i\in I},y(\rho_D))$.

The situation is slightly different in our case since we have to work with
$M_D$ instead of $R_D$. However, the only real difference between 
$R_D$ and $M_D$ is that the critical exponents in the singular representations
(\ref{eqRgen}) and (\ref{eqMgen}) are different, but the behaviour of the singularity
$\rho_D(t,(x_i)_{i\in I})$ is precisely the same. Note that after the integration step
we can set $t=1$. Now tightness for the normalised random vector that is encoded in the
function $M_D$ follows in the same way as for $R_D$, and since the singularity 
$\rho_D(1,(x_i)_{i\in I})$ is the same, we get precisely the same conditions as 
in the case of \cite[Theorem~3]{DGM}.

This means  we just have to check the above conditions for 
\[
F = F(1,z,(x_{2i})_{2i\in D},y ) =  z + z\sum_{2i\in D} x_{2i} {2i-1 \choose i} y^i,
\]
where all derivatives are evaluated at $z = \rho_D$, $x_{2i} = 1$, and $y = R_D(\rho_D) < 1/4$.
However, they are trivially satisfied since 
$\sum_{i\ge 1}  {2i-1 \choose i} i^K y^i  < \infty$
for all $K> 0$ and for positive real  $y < 1/4$.

\section{Central Limit Theorem for General Maps}\label{sec4.2}

We now assume that $D$ contains at least one odd number.
By Proposition~\ref{Pro2} we have to consider the
system of equations 
\begin{align*}
L_D &= z \sum_{i\in D} x_i \sum_m B_{i-2m-1,m} L_D^{i-2m-1} R_D^m,  \\
R_{D} &= tz+ z \sum_{i\in D} x_i \sum_m B_{i-2m-2,m}^{(+1)} L_D^{i-2m-2} R_D^{m+1},
\end{align*} 
for the generating functions
$L_D \equiv L_D(t,z,(x_i)_{i\in D})$ and
$R_D \equiv R_D(t,z,(x_i)_{i\in D})$, the generating function
\[
T_D \equiv T_D(t,z,(x_i)_{i\in D}) = 1 + \sum_{i\in D} x_i\sum_m \overline B_{i-2m,m} L_D^{i-2m} R_D^m
\]
and finally the generating function 
$M_D \equiv M_D(t,z,(x_i)_{i\in D})$ that satisfies the relation
\[
\frac{\partial M_D}{\partial t} = R_D/z-t + T_D.
\]

Again, if $D$ is finite, we can proceed as in the bipartite case by
applying \cite[Theorem 2.33, Lemma 2.27, and Theorem 2.25]{Drmotabook}
which implies the proposed central limit theorem.

If $D$ is infinite, we argue in a similar way as in the bipartite case.
The only difference is that we are not starting with one equation but with a system
of two equations that have the (general) form
\[
L = F(t,z,(x_i)_{i\in D},L,R), \quad R = G(t,z,(x_i)_{i\in D},L,R).
\]
Nevertheless, it is possible to reduce two equations of this form to a single one.
The proof of \cite[Theorem 2.33]{Drmotabook} shows that there are no analytic problems
since we have a positive and strongly connected system. We use the first equation
to obtain an implicit function solution $f = f(t,z,(x_i)_{i\in D},r)$ that satisfies
\[
f = F(t,z,(x_i)_{i\in D},f,r).
\]
Then we substitute $f$ for $L$ in the second equation and arrive at a single 
functional equation 
\[
R = G(t,z,(x_i)_{i\in D}, f(t,z,(x_i)_{i\in D},R),R)
\]
for $R = R_D(t,z,(x_i)_{i\in D})$. Note that the proof of \cite[Theorem 2.33]{Drmotabook}
assures that $f$ is analytic although $L$ and $R$ get singular.
Hence by setting 
\[
H(t,z,(x_i)_{i\in D},r) = G(t,z,(x_i)_{i\in D}, f(t,z,(x_i)_{i\in D},r),r)
\]
we obtain a single equation $R = H(t,z,(x_i)_{i\in D},R)$ for $R = R_D$ and 
we can apply the same method as in the bipartite case. Of course, the calculations
get more involved. For example, we have
\[
H_{x_i} = G_{x_i} + \frac{G_L F_{x_i}}{1-F_L}, 
\]
where
\begin{align*}
F_L &= \rho_D \sum_{i\in D} \sum_m (i-2m-1) B_{i-2m-1,m} L_0^{i-2m-2} R_0^m, \\
F_{x_i} &= \rho_D \sum_m B_{i-2m-1,m} L_0^{i-2m-1} R_0^m, \\
G_L &= \rho_D \sum_{i\in D} \sum_m (i-2m-2) B_{i-2m-2,m}^{(+1)} L_0^{i-2m-3} R_0^m, \\
G_{x_i} &= \rho_D \sum_m B_{i-2m-2,m}^{(+1)} L_0^{i-2m-2} R_0^m.
\end{align*}
From the proof of Lemma~\ref{Le3} we already know that $2L_0 - L_0^2 + 4R_0 < 1$, which implies that
\[
\sum_{i\ge 1}  \sum_m  m^K (i-2m-1) B_{i-2m-1,m} L_0^{i-2m-2} R_0^m < \infty
\]
for all $K> 0$. Furthermore, we have $F_L < 1$ and $G_R < 1$. Hence it follows that
$\sum_{i\in D} H_{x_i} < \infty$.
In the same way, we can handle the other conditions which completes the proof
of Theorem~\ref{Th1}.

\subsection{Weighted Maps}

In order to cope with weighted maps we just have to substitute $x_i = q_i$.
Then the coefficient $M_{\textbf{q},n} := [z^n] M(1,z,\textbf{q})$ is just the weighted sum of all maps
with $n$ edges. Actually, under the condition that $q_i = \Theta( i^\alpha)$ with
$\alpha \ge -\frac 32$ it follows that $M_{\textbf{q},n} \sim c\, n^{-5/2} \gamma^n$ for 
some positive constants $c,\gamma$. The reason is that we can show (almost in the same way as 
in Lemma~\ref{Le3}) that there exist solutions $L_0>0$, $R_0> 0$, $\rho> 0$ with
$2L_0 - L_0^2 + 4 R_0 < 1$ of the corresponding system (\ref{eqLe3x}). The simple argument is
that the series $\sum_{i\ge 1} \sqrt{i}\cdot i^\alpha$ diverges for $\alpha \ge -\frac 32$.
This proves that we have a square-root singularity for the functions $L$ and $R$, etc.

The central limit theorem can be proved also in the same way as above; we just have to
replace $x_i$ by $x_iq_i$. We leave the details to the reader.

\subsection{Mean and Covariance}

Recall that we have used \cite[Theorem 2.33, Lemma 2.27, and Theorem 2.25]{Drmotabook} to prove
the central limit theorem. Actually this method provides us also expressions for the constants
$\mu_d$ ($d\in D$) and a covariance matrix $\Sigma = (\sigma_{d_1,d_2})_{d_1,d_2\in D}$ for the limiting
Gaussian random variable $Z$. 
By \cite[Theorem 2.25]{Drmotabook} we have
\[
\mu_{d} = - \frac{\frac{\partial \rho_D}{\partial x_{d}}(1,{\bf 1})}{\rho_D(1,{\bf 1})}
\quad\mbox{and}\quad 
\sigma_{d_1,d_2} = \mu_{d_1}\mu_{d_2} + \delta_{d_1,d_2} \mu_{d_1}  
- \frac{\frac{\partial^2  \rho_D}{\partial x_{d}^2}(1,{\bf 1})}{\rho_D(1,{\bf 1})}.
\]
In particular, in the bipartite case we have (compare also with \cite[Theorem 2.23]{Drmotabook})
\[
\mu_{2j} = \frac{F_{x_{2j}}}{\rho_D(1,{\bf 1}) F_z}
\]
and
\begin{align*}
\sigma_{2i,2j} & = \mu_{2i}\mu_{2j} + \delta_{i,j} \mu_{2i} + \frac {1}{\rho_D(1,{\bf 1}) F_z^3F_{RR}}\Bigl( F_z^2(F_{RR}F_{x_{2i}x_{2j}} - F_{Rx_{2i}}F_{Rx_{2j}}) \\
&\qquad \qquad  -F_zF_{x_{2i}}(F_{RR}F_{zx_{2j}}-F_{Rz}F_{Rx_{2j}}) -F_zF_{x_{2j}}(F_{RR}F_{zx_{2i}}-F_{Rz}F_{Rx_{2i}}) \\
&\qquad \qquad + F_{x_{2i}}F_{x_{2j}}(F_{RR}F_{zz} - F_{Rz}^2) \Bigr),
\end{align*}
where $F(z,(x_{2i})_{2i\in D}, R) = z + z \sum_{2i\in D} x_{2i} {2i-1 \choose i} R^i$ and all
functions are evaluated at $z = z_0 = \rho_D(1,{\bf 1})$, $x_{2i} = 1$, $R = R_0$, and $z_0, R_0$ are defined as
in Lemma~\ref{Le1}:
\[
R_0 = F(z_0,{\bf 1}, R_0), \qquad 1 = F_R(z_0,{\bf 1}, R_0).
\]
Thus, (in the bipartite case) we get
\[
\mu_{2j} = z_0 {2j-1 \choose j} R_0^{j-1} \quad \mbox{and}\quad
\sigma_{2j,2k} = \mu_{2j} \delta_{j,k} - \mu_{2j}\mu_{2k} (1 + (j-1)(k-1) c)
\]
with $c = 1/(R_0 F_{RR})$.\footnote{Gregory Miermont has pointed out to the second author a very nice probabilistic
interpretation of these representations in terms of monotype Galton--Watson trees and infinite sequences of Gaussian random variables.}
We just have to observe the following relations:
\begin{gather*}
F_z = \frac {R_0}{z_0}, \quad F_{x_{2j}} = z_0 {2j-1 \choose j} R_0^j = R_0 \mu_{2j}, \quad 
F_{x_{2i}x_{2j}} = 0, \quad F_{zz} = 0, \\
F_{Rx_{2j}} =  j z_0 {2j-1 \choose j} R_0^{j-1} = j \mu_{2j}, \quad 
F_{zx_{2j}} = {2j-1 \choose j} R_0^j = \frac{R_0}{z_0} \mu_{2j}, \\
F_{zR} = \frac{F_R}{z_0} = \frac 1{z_0}.
\end{gather*}

In principle the same procedure also works for non-bipartite maps, however, the expressions
are much more involved. Therefore we only state the results for the basic case $D = \mathbb{N}$. 
The corresponding constants $\mu_d$ and $\sigma_{d_1,d_2}$ are given by
\[
\mu_d = A_d + 2 \overline A_d
\]
and
\begin{align*}
\sigma_{d_1,d_2} &= \mu_{d_1} \delta_{d_1,d_2} - \frac 32  \mu_{d_1}\mu_{d_2} + \frac{729}4 (d_1-1)({d_2}-1)A_{d_1-1}A_{{d_2}-1} \\
&+ \frac 92 \left( (\overline A_{d_1} + (d_1-1)\overline A_{d_1-1}) \overline A_{d_2} + (\overline A_{d_2} + ({d_2}-1)\overline A_{{d_2}-1}) \overline A_{d_1} \right) \\
&- \frac 1{18}\left( 39 \overline A_{d_1} + (d_1-1) \mu_{d_1-1} \right)\left( 39 \overline A_{d_2} + ({d_2}-1) \mu_{{d_2}-1} \right) \\
&-\frac 12 \left( \mu_{d_1} A_{d_2} + \mu_{d_2} A_{d_1}\right) + \frac 1{12} \left( ({d_2}-1) \mu_{d_1}  \mu_{{d_2}-1} + (d_1-1) \mu_{d_2} \mu_{d_1-1}    \right),
\end{align*}
where
\[
A_d = \frac 1{6^d} \sum_{m\ge 0} {d-1 \choose d-2m-1, m, m} 4^m
\]
and
\[
\overline A_d = \frac 1{6^d} \sum_{m\ge 0} {d-1 \choose d-2m-2, m,m+1 } 4^m.
\]

\section{Maps of Higher Genus}\label{sec5}

The bijection used in Section~\ref{sec2} relies solely on the orientability of the surface on which the maps are drawn. Therefore it can easily be extended to maps of higher genus, i.~e.\ embedded on an orientable surface of genus $g \in\mathbb{Z}_{>0}$ (while planar maps correspond to maps of genus $0$). The main difference lies in the fact that the corresponding mobiles are no longer trees but rather \emph{one-faced} maps of higher genus, while the other properties still hold. 

However, due to the appearance of cycles in the underlying structure of mobiles, another difficulty arises. Indeed, in the original bijection, vertices and edges in mobiles could carry labels (related to the geodesic distance in the original map), subject to local constraints. In our setting, the legs actually encode the \emph{local} variations of these labels, which are thus implicit. Local constraints on labels are naturally translated into local constraints on the number of legs. But the labels have to remain consistent along each cycle of the mobiles, which gives rise to non-local constraints on the repartition of legs.

In order to deal with these additional constraints, and to be able to control the degrees of the vertices at the same time, we will use a hybrid formulation of mobiles, namely $g$-mobiles, carrying both labels and legs. As before, we will focus on the simpler case of mobiles coming from bipartite maps.

\subsection{Definition of $g$-Mobiles}

Given $g \in\mathbb{Z}_{\ge 0}$, a \emph{$g$-mobile} is a one-faced map of genus $g$ -- embedded on the $g$-torus -- such that there are two kinds of vertices (black and white), edges only occur as black--black edges or black--white edges, and black vertices additionally have so-called ``legs'' attached to them (which are not considered to be edges), whose number equals the number of white neighbour vertices.

Furthermore, for each cycle $c$ of the $g$-mobile, let  $n_{\circ}(c)$, $n_{\rightarrow}(c)$ and $n_{\wn}(c)$ respectively be the numbers of white vertices on $c$, of legs dangling to the left (counterclockwise) of $c$ and of white neighbours to the left of $c$. One has the following constraint (see Figure~\ref{fig:cycle}):
\begin{equation}\label{prop:cycle} n_{\rightarrow}(c) = n_{\circ}(c) + n_{\wn}(c)\end{equation}
The \emph{degree} of a black vertex is the number of half-edges plus the number of legs that are 
attached to it.
A \emph{bipartite} $g$-mobile is a $g$-mobile without black--black edges.
A $g$-mobile is called \emph{rooted} if an edge is distinguished and oriented.\\
Notice that a $0$-mobile is simply a mobile as described in Definition~\ref{def:mobiles}.

\begin{figure}\label{fig:cycle}
\begin{minipage}{0.5\linewidth}
\includegraphics[width=0.95\linewidth]{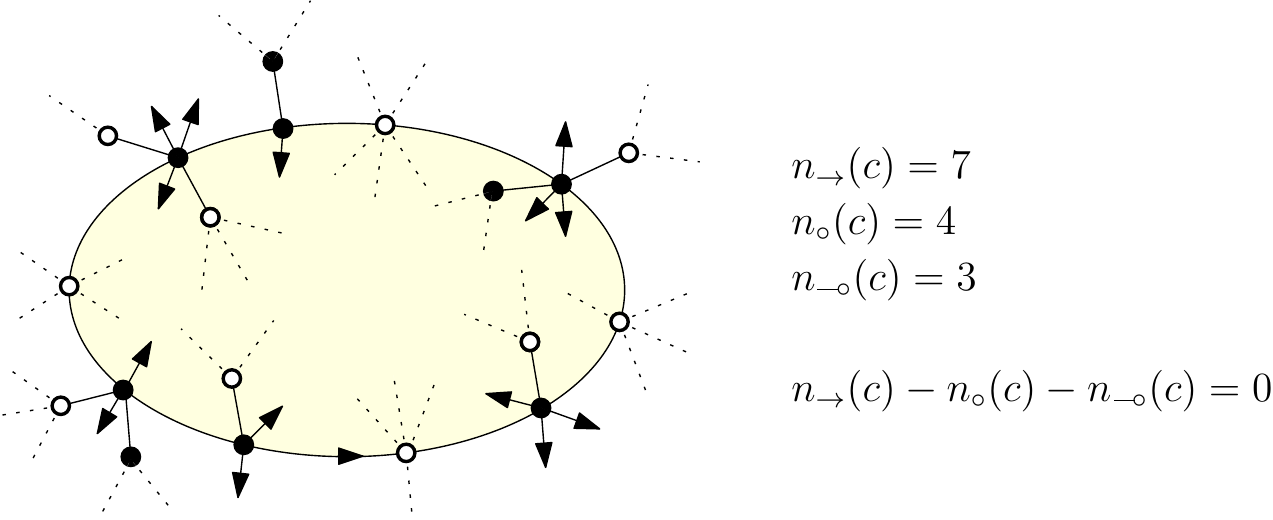}
\end{minipage}
\begin{minipage}{0.4\linewidth}
\caption{An oriented cycle $c$ in a $g$-mobile and the constraint on its left (coloured area). Notice that a similar constraint holds on its right, but is necessarily satisfied thanks to the local properties of a $g$-mobile.}
\end{minipage}
\end{figure}

%

Actually there is a direct analogue of Theorem~\ref{Th2}: $g$-mobiles are in bijection with pointed maps of genus $g$, with precisely the same properties stated in Theorem~\ref{Th2}.
This generalisation of the bijection to higher genus was first given by Chapuy, Marcus, and Schaeffer in~\cite{CMS07} for quadrangulations and by Chapuy in~\cite{Cha09} 
for Eulerian maps, from which we will exploit many ideas in this section.
%



\subsection{Schemes of $g$-Mobiles}

\begin{figure}[b]
\includegraphics[width=0.35\linewidth]{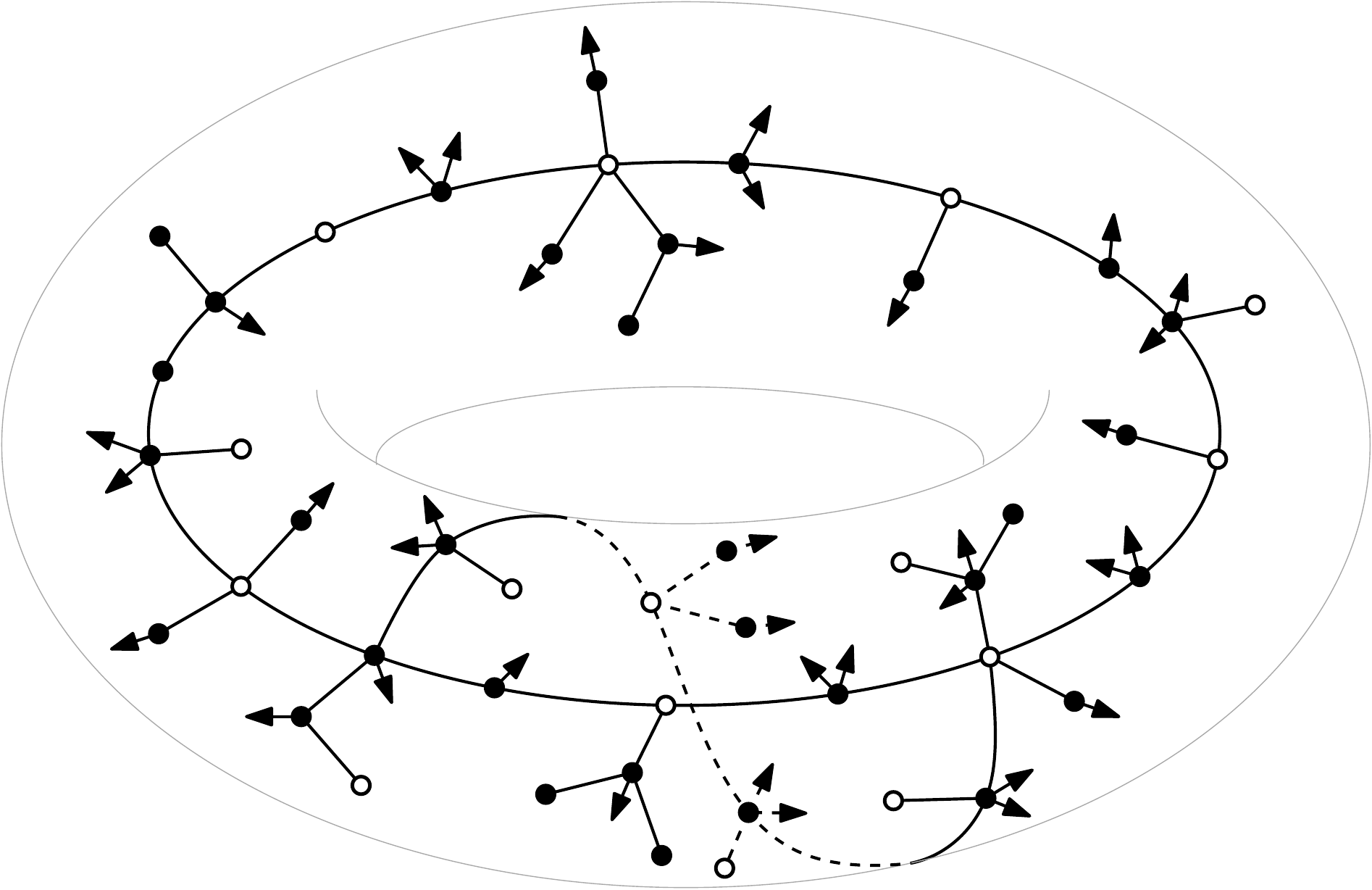} 
\hspace{1cm}
\includegraphics[width=0.35\linewidth]{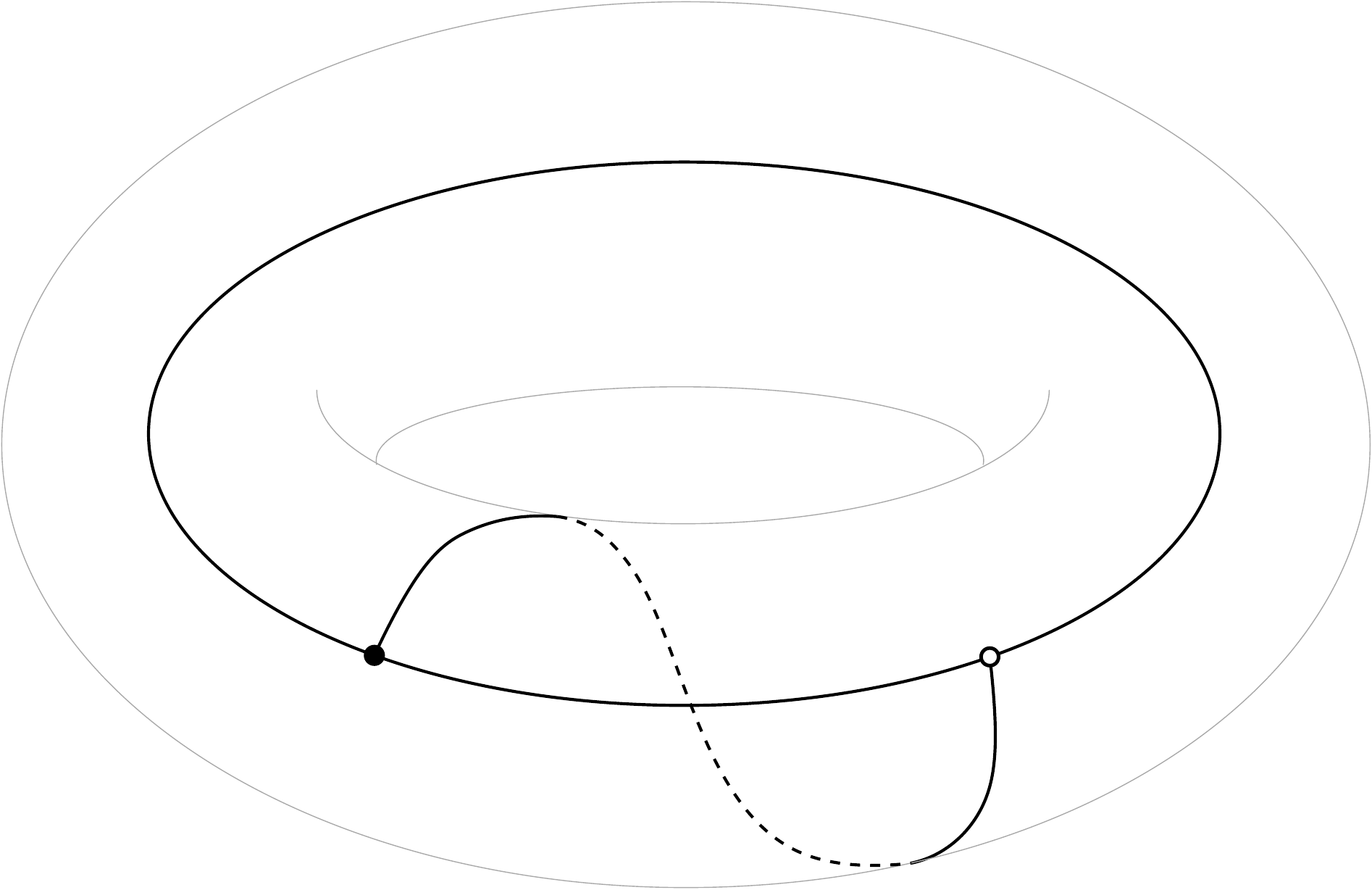}

\caption{A $1$-mobile on the torus and its scheme.}
\label{fig:g-mob}
\end{figure}

 A $g$-mobile is not as easily decomposed as a planar one, due to the existence of cycles. However, it still exhibits a rather simple structure, based on \emph{scheme extraction}.

The \emph{$g$-scheme} (or simply the \emph{scheme}) of a $g$-mobile is what remains when we apply the following operations (see Figure~\ref{fig:g-mob}): first remove all legs, then iteratively remove all vertices of degree $1$ and finally replace any maximal path of vertices of degree $2$ by a single edge.

Once these operations are performed, the remaining object is still a one-faced map of genus $g$, with black and white vertices (note that white--white edges can now occur), where the vertices have minimum degree $3$.

To count $g$-mobiles, one key ingredient is the fact that there is only a finite number of schemes of a given genus. Indeed, letting $e$, $v$ and $v_i$ ($i\ge 3$) be the number of edges, vertices and vertices of degree $i$ in a $g$-scheme, respectively, one gets: 
\begin{displaymath}
2e = 2(v + 2g -1) = \sum_{i\ge 3} i v_i \ge 3 \sum_{i\ge 3} v_i = 3v.
\end{displaymath}
The number of vertices (respectively edges) is then bounded by $4g-2$ (respectively $6g-3$), where this bound is reached for cubic schemes (see an example in Figure~\ref{fig:g-mob}).

To recover a proper $g$-mobile from a given $g$-scheme, one would have to insert a suitable planar mobile into each corner of the scheme and to substitute each edge with some kind of path of planar mobiles. Unfortunately, this cannot be done independently: Around each black vertex, the total number of legs in every corner must equal the number of white neighbours, and around each cycle, \eqref{prop:cycle} must hold.

In order to make these constraints more transparent, we will equip schemes with labels on white vertices and black corners. Now, when trying to reconstruct a $g$-mobile from a scheme, one has to ensure that the local variations are consistent with the global labelling. To be precise, the label variations are encoded as follows (see Figure~\ref{fig:labels}):

\begin{figure}
\includegraphics[width=0.7\linewidth]{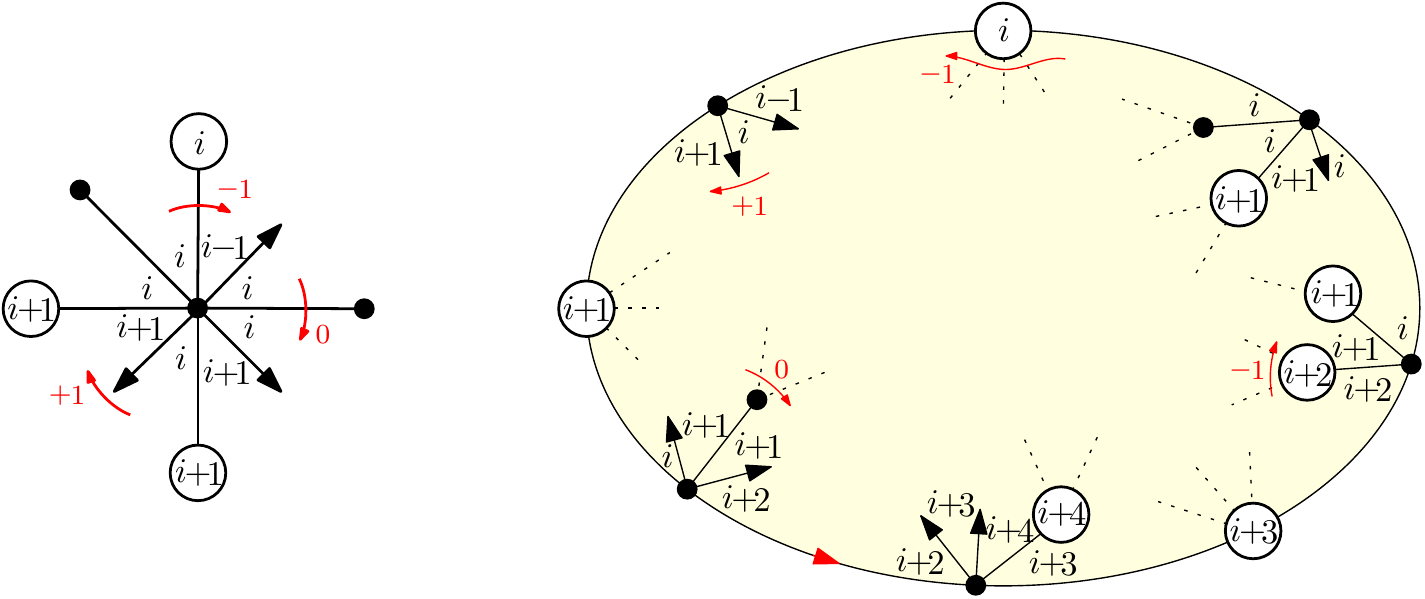}
\caption{The variations of labels around a black vertex and along an oriented cycle.}
\label{fig:labels} 
\end{figure}

\begin{itemize}
\item Around a black vertex of degree $d$, let $(l_1,\dots,l_d)$ be the labels of its corners read in clockwise order. For all $i$,
\begin{displaymath}l_{i+1}-l_{i} = 
\left\lbrace
\begin{array}{cl}
+1  & \mbox{if there is a leg between the two corresponding corners,} \\
0 & \mbox{if there is a black neighbour,}\\
-1 & \mbox{if there is a white neighbour.}
\end{array}\right.
\end{displaymath}
\item Along the left side of an oriented cycle, the label decreases by $1$ after a white vertex or when encountering a white neighbour and increases by $1$ when encountering a leg.
\end{itemize}

The above statements hold for general -- as well as bipartite -- mobiles. In the following, we will only consider bipartite mobiles, as they are much easier to decompose.

\subsection{Reconstruction of Bipartite Maps of Genus $g$}

In the following, it will be convenient to work with rooted schemes. One can then define a canonical labelling and orientation for each edge of a rooted scheme. An edge $e$ now has an origin $e_{-}$ and an endpoint $e_{+}$. The $k$ corners around a vertex of degree $k$ are ordered clockwise and denoted by $c_1,\dots,c_k$.\\

Given a scheme $S$, let $V_{\circ}, V_{\bullet}, C_\circ, C_\bullet$ be the sets of white and black vertices and of white and black corners, respectively. A \emph{labelled scheme} $(S,(l_c)_{c\in V_\circ \cup C_\bullet})$ is a pair consisting of a scheme $S$ and a labelling on white vertices and black corners, with $l_c\ge 0$ for all $c$. Labellings are considered up to translation, as those will not affect local variations. For an edge $e\in{E_S}$ of $S$, we associate a label to each extremity $l_{e_-}, l_{e_+}$. If an extremity is a white vertex of label $l$, its label is $l$. If the extremity is a black vertex, its label is the same as the next clockwise corner of the black vertex.

Let a \emph{doubly-rooted planar mobile} be a rooted (on a black or white vertex) planar mobile  with a secondary root (also black or white). These two roots are the extremities of a path $(v_1,\dots,v_k)$. The \emph{increment} of the doubly-rooted mobile is then defined as $n_{\rightarrow}-n_{\circ}-n_{\wn}$, which is not necessarily 0, as the path is not a cycle.




Similar to~\cite{Cha09}, we present a non-deterministic algorithm to reconstruct a $g$-mobile:\\

\noindent
\textbf{Algorithm.}

	(1) \emph{Choose a labelled $g$-scheme $(S,(l_c)_{c\in V_\circ \cup C_\bullet})$.}

	(2) \emph{For all $v\in V_\bullet$, choose a sequence of non-negative integers $(i_k)_{1\le k\le deg(v)}$, then attach $i_k$ planar mobiles and $i_k+l_{c_{k+1}}-l_{c_{k}}+1$ legs to $c_k$ (the $k^{th}$ corner of $v$).}

	(3) \emph{For all $e\in S$, replace $e$ by a doubly-rooted mobile of increment
	\begin{displaymath}
		\Delta(e)=l_{e_+}-l_{e_-} + \left\lbrace\begin{array}{cl}
			+1  & \mbox{if $e_-$ is white,} \\
			-1 & \mbox{if $e_-$ is black.}
		\end{array}\right.
	\end{displaymath}}

	(4) \emph{On each white corner of $S$, insert a planar mobile.}

	(5) \emph{Distinguish and orient an edge as the root.}

\begin{figure}
\includegraphics[width=0.32\linewidth]{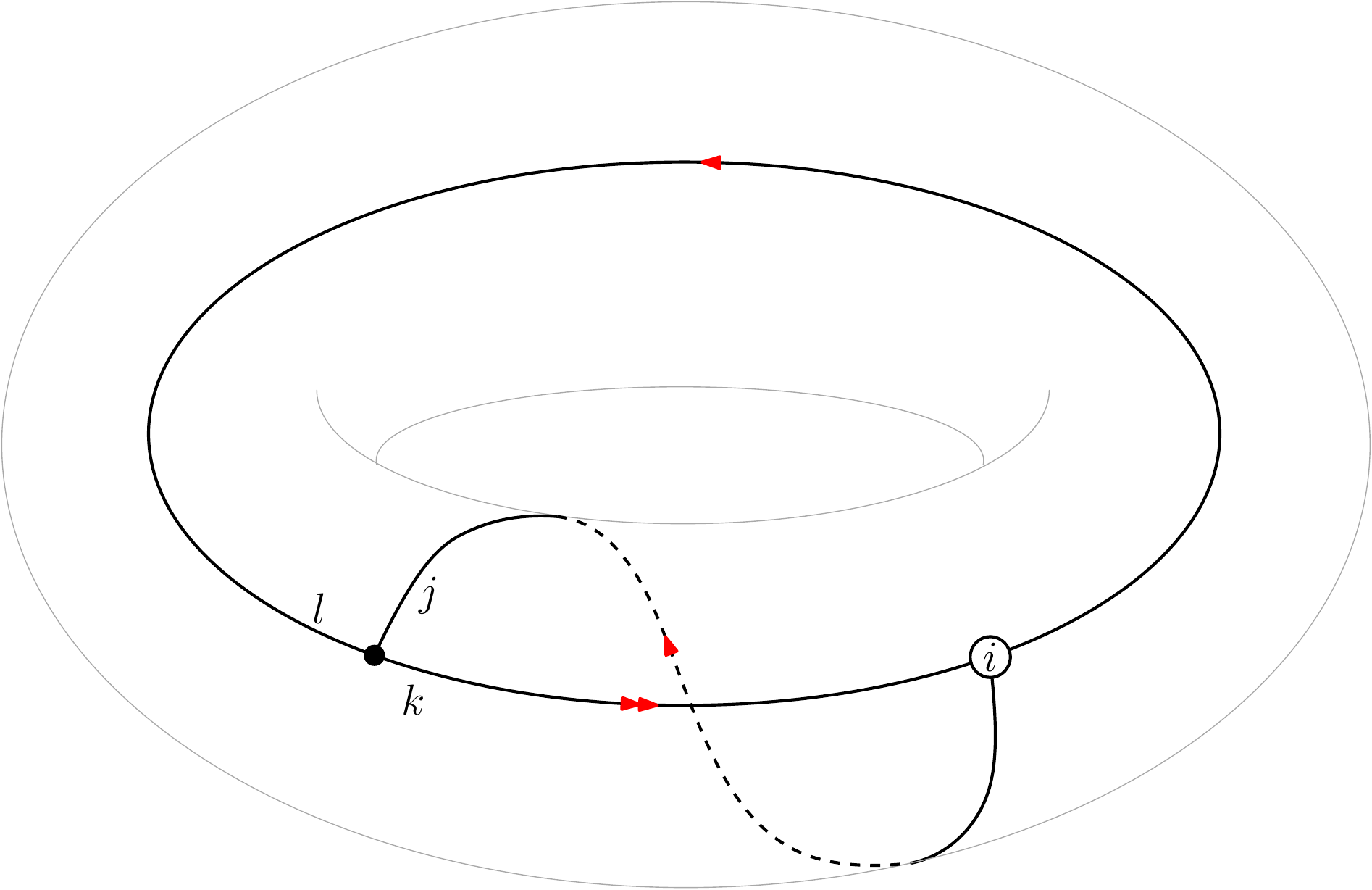}
\includegraphics[width=0.32\linewidth]{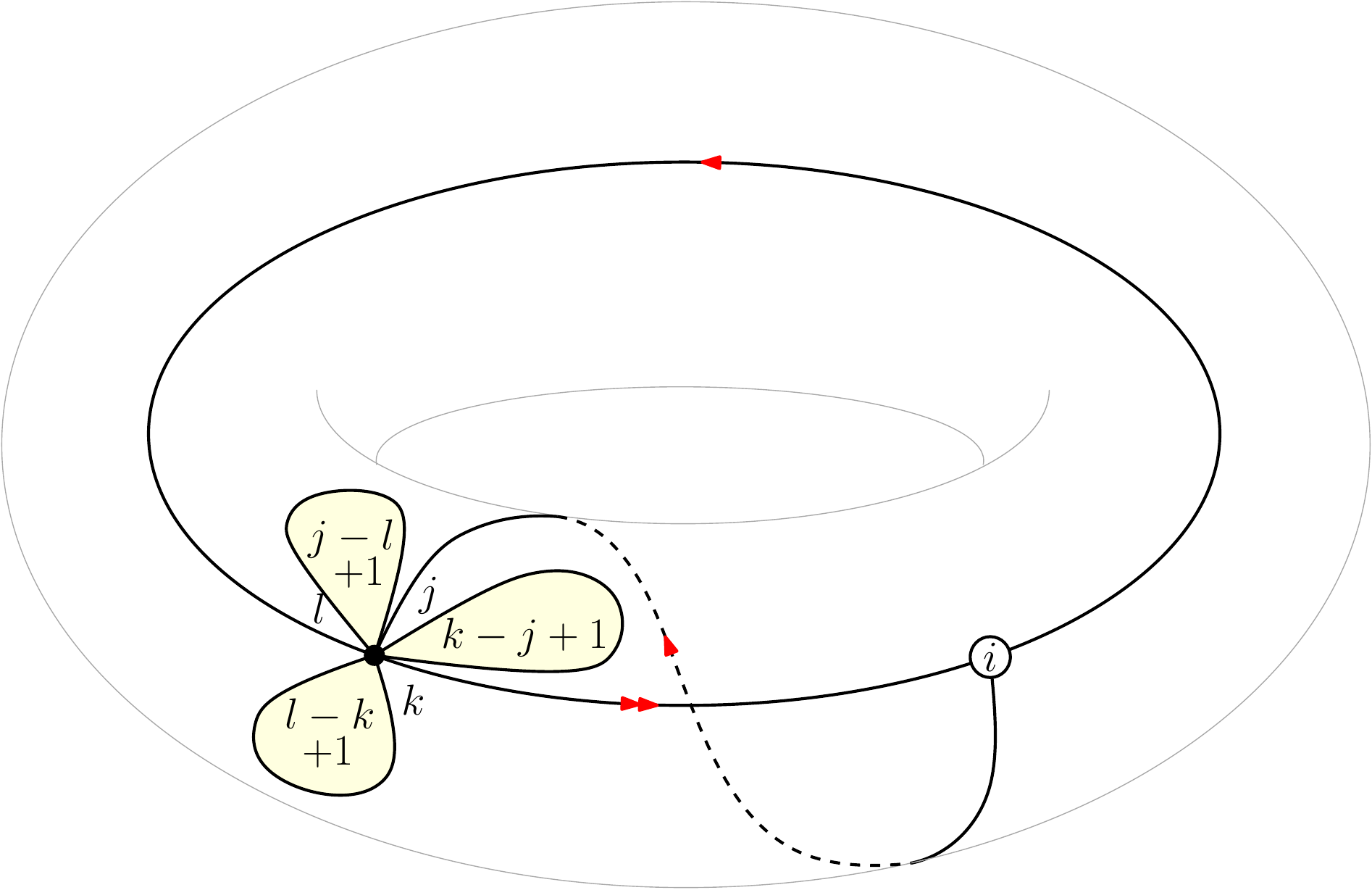}
\includegraphics[width=0.32\linewidth]{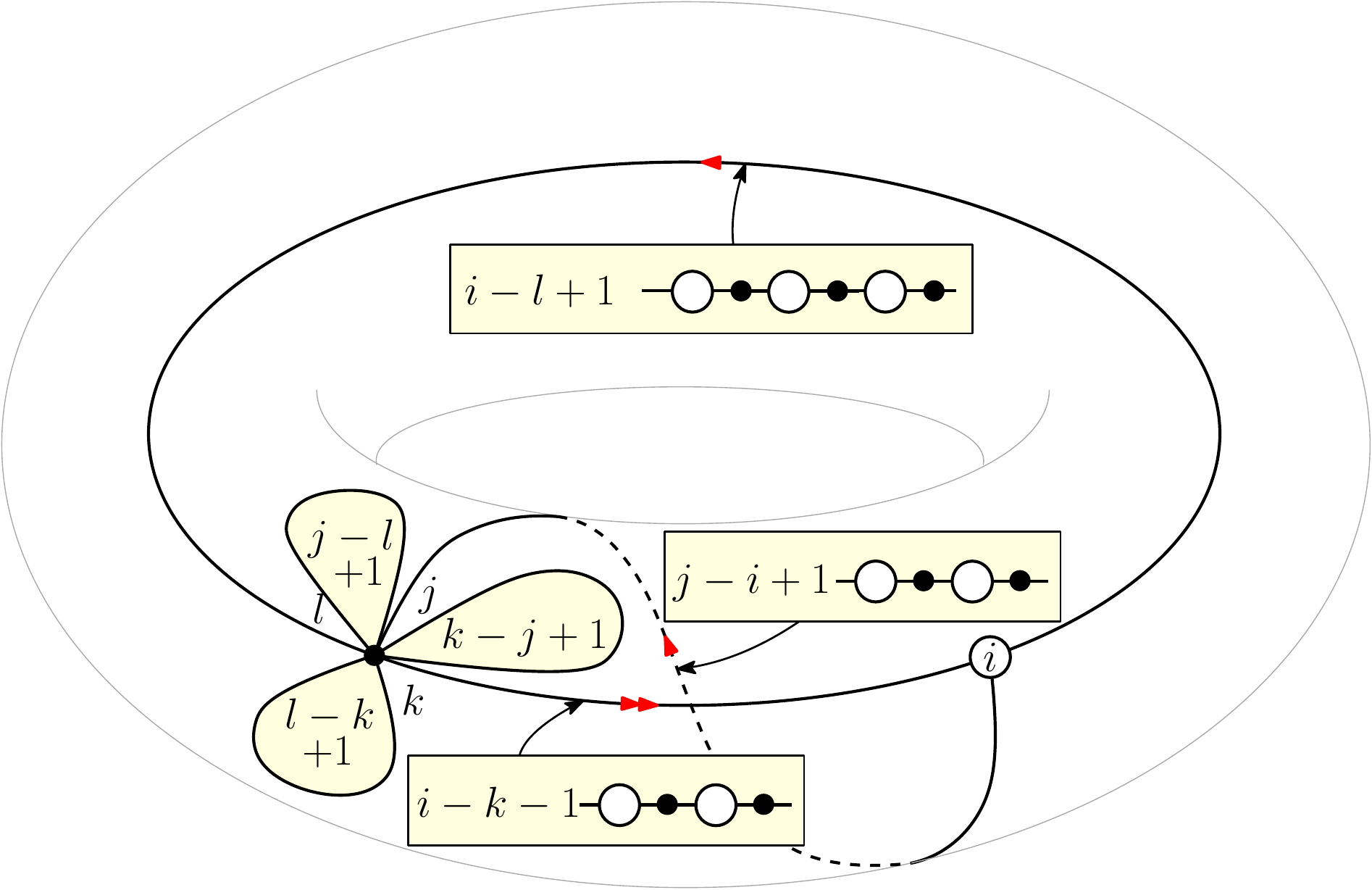}
\caption{Steps (1)--(3) of the algorithm.}
\label{fig:algo} 
\end{figure}

\begin{proposition}\label{pro:g-mob-algo}
Given $g>0$, the algorithm generates each rooted bipartite $g$-mobile whose scheme has $k$ edges in exactly $2k$ ways.
\end{proposition} 

\begin{proof}
One can easily see that the obtained object is indeed bipartite. Attaching planar mobiles and legs added at step (2) in a corner $c_k$ creates new corners, such that:
\begin{itemize}
\item The first carries the same label $l_{c_k}$ as $c_k$, and
\item the last carries the label $l_{c_k} + (i_k + l_{c_{k+1}} - l_{c_{k}} +1) - i_k = l_{c_{k+1}}+1$.
\end{itemize}
The next corner should then be labelled $(l_{c_{k+1}}+1)-1 = l_{c_{k+1}}$, due to the next white neighbour, which is precisely what we want.

In the same fashion, at step (3), a simple counting shows that each edge is replaced by a path such that the labels along it evolve according to the scheme labelling.

We thus obtain a well-formed rooted bipartite $g$-mobile, with a secondary root on its scheme. Since the first root destroys all symmetries, there are exactly $2k$ choices for the secondary root which would give the same rooted $g$-mobile.
\end{proof}

\subsection{Bipartite $g$-Mobile Counting}

Recall that, in the bipartite case, the generating series for rooted planar mobiles $R \equiv R(z,t,(x_{2i})_{i\ge 1})$ satisfies Equation~\eqref{eqwb}:

\begin{equation}
R/z - t = \sum_{i\ge 1} x_{2i} {2i-1 \choose i} R^i.
\end{equation}

A $g$-mobile can now be uniquely decomposed as a scheme where each edge is substituted by a sequence of elementary cells. By definition of a $g$-mobile, one needs to track the \emph{increment}, i.~e.\ the variation of labels along it, of each cell to ensure that the overall cycle constraints are satisfied.

An \emph{elementary cell} is a half-edge connected to a black vertex itself connected to a white vertex with a dangling half-edge. The white vertex has a sequence of black-rooted mobiles attached on each side. For an elementary cell of increment $i$, the black vertex has $k\ge0$ white-rooted mobiles and $k+i+1\ge0$ legs on its left, $l\ge 0$ white-rooted mobiles and $l-i+1$ legs on its right, and its degree is $2(k+l+2)$.  

The generating series $P \equiv P(t,z,R,(x_{2i}),s)$ of a cell, where $s$ marks the increment, is:
\begin{displaymath}P(t,z,R,(x_{2i}),s)= \frac{R^2}{t}\sum_{k,l\ge 0} x_{2(k+l+2)} R^{k+l} \sum_{i=-k-1}^{l+1} {2k+i+1 \choose k} {2l -i +1 \choose l} s^{i}.
\end{displaymath}

Depending on the edge end colours, there might be an additional black or white vertex inserted at the end of the sequence of elementary cells. This is reflected by an extra factor in the generating series $S_e \equiv S_e(t,z,(x_{2i}),s)$:
\begin{displaymath}
S_{(u,v)}(t,z,(x_{2i}),s)=\left\lbrace\begin{array}{cl}
\frac{1}{1-P} & \mbox{if $(u,v)=(\circ,\bullet)$ or $(\bullet,\circ)$,} \\
\frac{stz}{R^2}(\frac{1}{1-P}-1) & \mbox{if $(u,v)=(\circ,\circ)$,} \\
\frac{R^2}{stz}\frac{1}{1-P} & \mbox{if $(u,v)=(\bullet,\bullet)$.}
\end{array}\right.
\end{displaymath}

\begin{figure}
\centering
\includegraphics[width=0.5\linewidth]{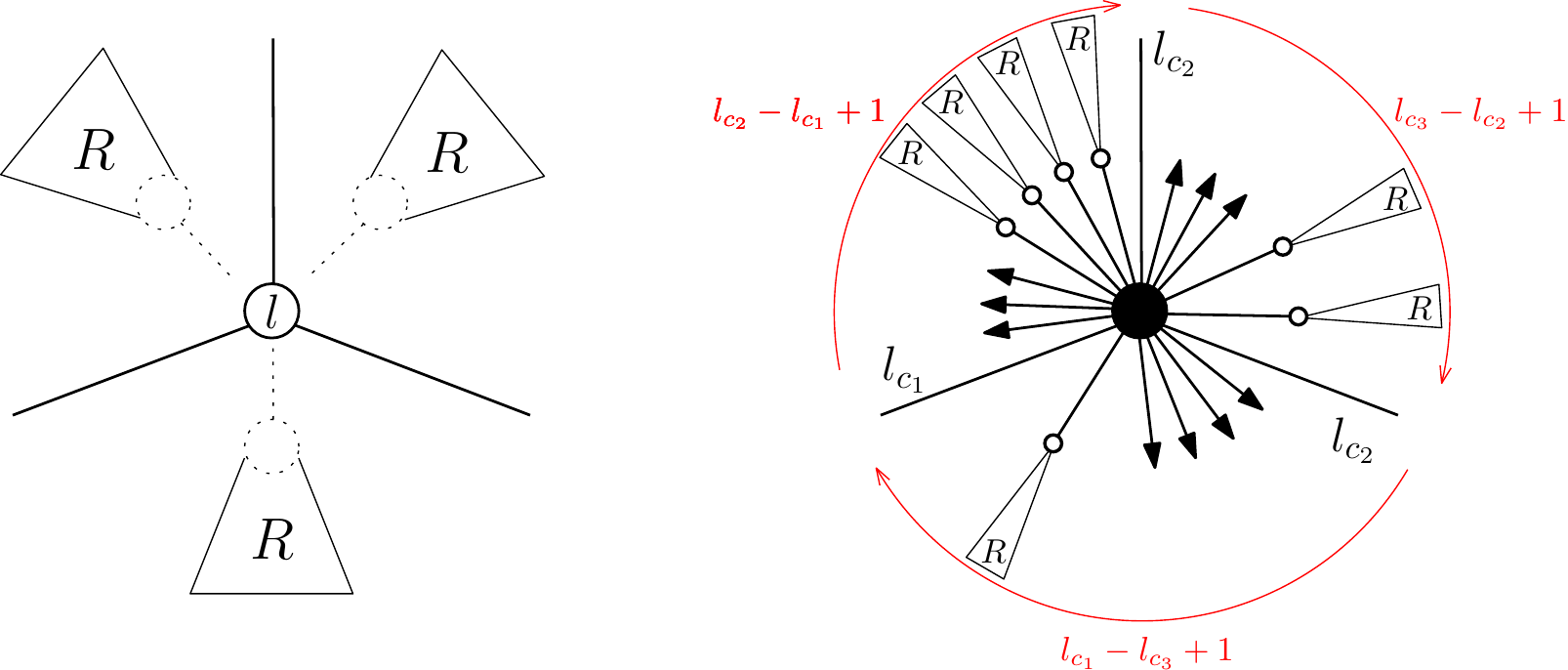}
\caption{Decorating white and black vertices of a scheme.}
\label{fig:subs-vertices} 
\end{figure}

Finally, exploiting steps $(2)$ and $(4)$ in the algorithm of Proposition~\ref{pro:g-mob-algo}, the vertices of the scheme are also decorated in the following way (see Figure~\ref{fig:subs-vertices}). To each white corner is attached a rooted planar mobile, counted by:
\[
	\frac{R}{zt},
\]
while for each black vertex $v$ of degree $d$, with corners $c_1,\dots,c_d$, a sequence of legs and mobiles is attached to the corner $c_k$ ($1\le k\le d$), such that the label variation around $c_k$ equals $l_{c_k+1}-l_{c_k}+1$, which is counted by:
\[
	V_{v} \equiv V_{v}(z,t,(x_{2i})) := \sum_{i_1,\dots,i_d\ge 0} \left(\prod_{k=1}^{d} {2i_k + l_{c_{k+1}} - l_{c_k} +1 \choose i_k} R^{i_k}\right) x_{2(d +\sum i_k)}.
\] 

We can now express the generating series $Q_{S} \equiv Q_S(t,z,R,(x_{2i}))$ of rooted bipartite $g$-mobiles with scheme $S$, :
\begin{equation}\label{eq:gmob}
Q_S = 2\frac{z\partial}{\partial z} \frac{1}{2|E|} z^{|E|}t^{|V_\circ|}\left(\frac{R}{tz}\right)^{|C_\circ|} \sum_{(l_c)\textrm{ labelling}}  \left[ \prod_{e\in E} [s^{\Delta(e)}] S_{(e_-,e_+)} \prod_{v\in V_\bullet} V_v  \right].
\end{equation}

\begin{proposition}
The generating series $M_D^{(g)}:=M_D^{(g)}(t,z,(x_{2i}))$ for the family of rooted bipartite maps of genus $g$, where the vertex degrees belong to $D$, satisfies the relation:
\begin{equation} \frac{\partial M_D^{(g)}}{\partial t} = \frac{2}{z}\sum_{\substack{S \textrm{ scheme} \\ \textrm{of genus } g}} Q_S(t,z,(x_{2i}\mathbb{1}_{\{2i\in D\}})).
\end{equation}
\end{proposition}

\begin{proof}
This follows directly from the bijections between $g$-mobiles and maps of genus $g$  and equation~\eqref{eq:gmob}, and by noticing that, as in the planar case, the equation can be specialised to constrained degrees by setting the variables $x_{2i} := 0$ when $2i$ does not belong to $D$.
\end{proof}

\subsection{Asymptotics of $g$-Mobiles}

We proceed similarly to \cite{Cha09}. However, for the sake of brevity we will not
work out all technical details. For example, we will take only care of the (local) singular expansion
and restrict ourselves to the case $\overline d = 1$.

First we need proper expansion of the coefficients of $(1-P)^{-1}$. 
\begin{lemma}\label{lem:asymp-S}
We have, as $|\Delta| \to \infty$, 
$$
[s^{\Delta}] \frac{1}{1-P(t,z,R(t,z),(x_{2i});s)} = C_{{\rm sgn}\, \Delta} (t,z)\alpha_{{\rm sgn}\, \Delta}(t,z)^{|\Delta|}
+ \bigo\left(|\alpha_{\pm}(t,z)| - \delta\right)^{|\Delta|},
$$
where $\alpha_{\pm}(t,z) \sim 1 - c_1\left(1-{z}/{\rho}\right)^{1/4}$ and $C_{\pm}(t,z)\sim c_2 \left(1-{z}/{\rho}\right)^{-1/4}$ 
for some positive constants $c_1, c_2$ and $\rho \equiv \rho(t,(x_{2i}))$
\end{lemma}

\begin{proof}
With the help of \eqref{eqR0} it is easy to check that the following three relations hold when
we evaluate at $t$ close to $1$, $x_{2i}$, $2i\in D$, close to $1$, $R = R_0(t,(x_{2i}))$, 
$z = \rho_D  \equiv \rho_D(t,(x_{2i}))$, and $s = 1$:
$$
P = 1, \quad P_s = 0, \quad P_{ss} \ne 0, \quad P_R \ne 0.
$$
Thus we have locally two solutions $s = \alpha_{1,2}(t,z)$ of the equation
\[
P(t,z,R(t,z),(x_{2i}),s) = 1
\]
that are of the form
$\alpha_{1,2}(t,z) = 1 \mp c_1\left(1-{z}/{\rho_D}\right)^{1/4} + \bigo\left(\left(1-{z}/{\rho_D}\right)^{1/2} \right)$.
For $s$ with $|\alpha_{1}(t,z)| < |s| < |\alpha_{2}(t,z)|$ we also have $|P|<1$ and consequently by Cauchy integration 
applied to the Laurent series $s\mapsto P$
\[
[s^{\Delta}] \frac{1}{1-P(t,z,R(t,z),(x_{2i});s)} = \frac 1{2\pi i} \int_{|s| = s_0} \frac{1}{1-P(t,z,R(t,z),(x_{2i});s)} \frac{ds}{s^{\Delta+1}},
\]
where $|\alpha_{1}(t,z)| < s_0 < |\alpha_{2}(t,z)|$. Clearly $s = \alpha_{1,2}(t,z)$ are polar singularities of $1/(1-P)$. Thus, if we shift
the integral to a circle $|s| = |\alpha_{2}(t,z)| + \delta$ (for some $\delta > 0$) and by collecting the residue at 
$s = \alpha_{2}(t,z)$, we get, as $\Delta \to + \infty$,
\[
[s^{\Delta}] \frac{1}{1-P(t,z,R(t,z),(x_{2i});s)} = C_2(t,z) \alpha_2(r,z)^{-\Delta} + \left(|\alpha_{2}(t,z)| + \delta\right)^{-\Delta},
\]
where $C_2(t,z) = 1/P_s(t,z,R(t,z),(x_{2i})),\alpha_2(t,z)) = c_2\left(1-{z}/{\rho_D}\right)^{-1/4} + \bigo(1)$.
Similarly we obtain the corresponding expansion for $\Delta \to - \infty$. 
Thus, setting $\alpha_+(t,z) = \alpha_2(t,z)^{-1}$, $\alpha_-(t,z) = \alpha_1(t,z)$, 
$C_{+}(t,z) = C_{2}(t,z)$, and $C_{-}(t,z) = C_{1}(t,z)$ completes the proof of the lemma.
\end{proof}

%


With the help of these preliminaries we can determine the singular structure of the generating functions $Q_S(t,z,(x_{2i}))$
related to a scheme $S$. For the sake of brevity we will only discuss labelled schemes where all vertices are white. Thus all edges 
are white--white and labels are carried by the white vertices. 
Without loss of generality, one can assume that the minimal label is 0 (by shifting all labels, as only the differences matter).

Recall the expression of $Q_S$, from Equation~\eqref{eq:gmob}, when $S$ only has white vertices:
\[
	Q_S(t,z,(x_{2i})) = \frac{z\partial}{\partial z} \frac{1}{|E|} z^{|E|}t^{|V_\circ|}\left(\frac{R}{tz}\right)^{|C_\circ|} \sum_{(l_c)\textrm{ labelling}}  \prod_{e\in E} [s^{\Delta(e)}] S_{(e_-,e_+)}.
\]

In order to handle the sums over all labellings, define $\lambda : V_\circ \rightarrow [|0,M|]$ (where $M=\textrm{card}(\{\textrm{labels of } V_\circ\})-1$), the relative order of the labels. Labels can then be rewritten as: $$\forall v\in V_\circ, l_v= \sum_{i=1}^{\lambda(v)} \delta_i, \qquad \textrm{with } \delta_i \in\mathbb{Z}_{>0}.$$

Hence we can rewrite the sum as follows, using the asymptotics of Lemma~\ref{lem:asymp-S}:
\begin{align*}
	\sum_{(l_c)\textrm{ labelling}}  \prod_{e\in E} [s^{\Delta(e)}] S_{(e_-,e_+)} &= \sum_{\delta_1,\dots,\delta_M >0}  \prod_{e\in E} \frac{tz}{R^2}[s^{\sum_j A_{e,j}\delta_j}] \left(\frac{1}{1-P}-1\right) \\
	&\sim \left(\frac{tz}{R^2}\right)^{|E|} \sum_{\delta_1,\dots,\delta_M >0} \prod_{e\in E} C_+(t,z)\alpha_+(t,z)^{\sum_j A_{e,j}\delta_j} \\ 
	&\sim \left(\frac{tzC(t,z)}{R^2}\right)^{|E|} \prod_{j=1}^M  \frac{ \prod_e \alpha_+(t,z)^{A_{e,j}} }{1- \prod_e \alpha_+(t,z)^{A_{e,j}}}
\end{align*}

Finally, we obtain that:
\begin{align*}
Q_S(t,z,(x_{2i})) &\sim \frac{z\partial}{\partial z} \frac{1}{|E|} z^{|E|}t^{|V_\circ|}\left(\frac{R}{tz}\right)^{|C_\circ|}  \left(\frac{tzC(t,z)}{R^2}\right)^{|E|} \prod_{j=1}^M  \frac{ \prod_e \alpha_+(t,z)^{A_{e,j}} }{1- \prod_e \alpha_+(t,z)^{A_{e,j}}}  \\
&\sim \frac{z\partial}{\partial z} \frac{1}{|E|} t^{|V_\circ|-|E|} C(t,z)^{|E|}  \frac{ 1 }{(1- \alpha_+(t,z))^M} \\
&\sim c_3\frac{z\partial}{\partial z} \frac{1}{|E|} \left(1-\frac{z}{\rho_D}\right)^{(-|E| - M)/4}
\end{align*}

The main contribution will then come from cubic schemes with maximal $M$, i.~e.\ where all labels are distinct. Thus $|E|=6g-3, M=|V|-1=4g-3$.
$$
Q_S(t,z,(x_{2i})_{2i\in D}) \sim c_3\frac{z\partial}{\partial z} \frac{1}{|E|} \left(1-\frac{z}{\rho_D(t,(x_{2i})_{2i\in D})}\right)^{-5g/2+3/2}
$$

Similar asymptotics can be derived --- with more technical computations --- for the mobiles where the scheme also has black vertices. 

Summing up over all the dominant schemes of genus $g$, and after an integration step, we recover the expected singular behaviour
$$
M^{(g)}_D(t,z,(x_{2i})_{2i\in D}) \sim c_4\frac{z\partial}{\partial z} \frac{1}{2|E|} \left(1-\frac{z}{\rho(t,(x_{2i})_{2i\in D})}\right)^{-5g/2+5/2}
$$
which corresponds to the asymptotics given in Theorem~\ref{Th3} (when we set $t=1$ and $x_{2i} = 1$, $2i \in D$). 
The central limit theorem follows as in the planar case by varying $x_{2i}$ around 1.

As a final note, an expression of the same flavour as Equation~\eqref{eq:gmob} can be derived for $g$-mobiles coming from non-bipartite maps. However, the expression becomes much more involved and it seems quite difficult to extract asymptotics, though it should definitely have the same shape.

\medskip\noindent
{\bf Acknowledgements.} The authors are very grateful to Mireille Bousquet-Melou, Guillaume Chapuy, and Gr\'egory Miermont for their help and for several valuable remarks. We also thank the anonymous reviewers for their helpful suggestions and comments.

\bibliographystyle{plain}
\bibliography{Bibliography}

\appendix

\end{document}